\documentclass[a4paper,12pt]{article}

\usepackage{amssymb,amsmath,latexsym,amsthm}
\usepackage[english]{babel}

\newtheorem{Thm}{Theorem}[section]
\newtheorem{Lem}[Thm]{Lemma}
\newtheorem{Prop}[Thm]{Proposition}
\newtheorem{Cons}[Thm]{Construction}
\newtheorem{Ass}[Thm]{Assumption}
\newtheorem{Alg}[Thm]{Algorithm}

\newtheorem{Cor}[Thm]{Corollary}

\theoremstyle{definition}
\newtheorem{Def}[Thm]{Definition}
\newtheorem{Rem}[Thm]{Remark}
\newtheorem{Ex}[Thm]{Example}

\numberwithin{equation}{section}

\def\into{\hookrightarrow}

\def\longinto{\lhook\joinrel\longrightarrow}

\newbox\mybox

\def\ontoover#1{\mathrel{
       \setbox\mybox=\hbox spread 1.4em{\hfil$\scriptstyle#1$\hfil}
       \vbox{\offinterlineskip\copy\mybox
             \hbox to\wd\mybox{\rightarrowfill\hskip-2.8mm
                               $\rightarrow$}}}}

\newcommand{\BA}{{\mathbb{A}}}

\newcommand{\BF}{{\mathbb{F}}\,\!{}}

\newcommand{\BN}{{\mathbb{N}}}

\newcommand{\BP}{{\mathbb{P}}}
\newcommand{\BQ}{{\mathbb{Q}}}

\newcommand{\BZ}{{\mathbb{Z}}}

\newcommand{\Fn}{{\mathfrak{n}}}

\newcommand{\Fp}{{\mathfrak{p}}}

\newcommand{\Fr}{{\mathfrak{r}}}

\newcommand{\MCD}{{\mathcal{D}}}

\newcommand{\MCG}{{\mathcal{G}}}

\newcommand{\MCO}{{\mathcal{O}}}

\newcommand{\MCS}{{\mathcal{S}}}
\newcommand{\MCT}{{\mathcal{T}}}

\newcommand{\MCX}{{\mathcal{X}}}
\newcommand{\MCY}{{\mathcal{Y}}}

\DeclareMathOperator{\disc}{disc}
\DeclareMathOperator{\rdisc}{rdisc}
\DeclareMathOperator{\Ver}{V}
\DeclareMathOperator{\Edg}{E}
\DeclareMathOperator{\Nbs}{Nbs}
\DeclareMathOperator{\degree}{deg}
\DeclareMathOperator{\trd}{trd}
\DeclareMathOperator{\nrd}{nrd}

\DeclareMathOperator{\odd}{odd}
\DeclareMathOperator{\one}{one}

\DeclareMathOperator{\diam}{diam}
\DeclareMathOperator{\sgn}{sgn}

\newcommand{\GL}{{\rm GL}}
\newcommand{\SL}{{\rm SL}}
\newcommand{\PGL}{{\rm PGL}}

\newcommand{\CT}{\mathcal{T}}
\newcommand{\CX}{\mathcal{X}}
\newcommand{\CY}{\mathcal{Y}}
\newcommand{\CS}{\mathcal{S}}
\newcommand{\CC}{\mathcal{C}}
\newcommand{\CG}{\mathcal{G}}
\newcommand{\CO}{\mathcal{O}}
\newcommand{\eps}{\varepsilon}
\newcommand{\Aa}{A}

\newcommand{\Stab}{\mathop{\rm Stab}\nolimits}

\newcommand{\End}{\mathop{\rm End}\nolimits}
\newcommand{\Hom}{\mathop{\rm Hom}\nolimits}

\def\Leg#1#2{\left(\frac{#1}{#2}\right)} 				
\def\QuatAlg#1#2#3{\left(\frac{#1,#2}{#3}\right)}		

\newcommand{\BFq}{{\BF_q}}
\newcommand{\Ulambda}{\underline{\lambda}}
\newcommand{\Ux}{\underline{x}}
\DeclareMathOperator{\EP}{PE}
\DeclareMathOperator{\Frob}{Frob}
\DeclareMathOperator{\vnf}{vnf}

\begin{document}

\title{On computing quaternion quotient graphs for function fields}
 
\author{\sc Gebhard B\"ockle, Ralf Butenuth}

\maketitle

\section*{R\'{e}sum\'{e}}
Soit $\Lambda$ un  $\BFq[T]$-ordre maximal d'un corps de quaternions sur $\BFq(T)$ non-ramifi\'e \`a la place $\infty$. Cet article donne un algorithme pour calculer un domaine fondamental de l'action  de groupe des unit\'es $\Lambda^*$ sur l'abre de Bruhat-Tits $\CT$ associ\'ee \`a $\PGL_2(\BFq((1/T)))$, l'action qui est un analogue en corps de functions de l'action d'un groupe cocompact Fuchsian sur le demi-plan superieur. L'algorithme donne \'egalement une pr\'esentation explicit du groupe $\Lambda^*$ par g\'en\'erateurs et relations. En outre nous trouvons une borne sup\'erieure pour le temps de calcul en utilisant que le graphe quotient $\Lambda^*\backslash\CT$ est {\em pr\`esque} Ramanujan.

\section*{Abstract}
Let $\Lambda$ be a $\BFq[T]$-maximal order in a division quaternion algebra over $\BFq(T)$ which is split at the place $\infty$.  The present article gives an algorithm to compute a fundamental domain for the action of the group of units $\Lambda^*$ on the Bruhat-Tits tree $\CT$ associated to $\PGL_2(\BFq((1/T)))$. This action is a function field analog of the action of a co-compact Fuchsian group on the upper half plane. The algorithm also yields an explicit presentation of the group $\Lambda^*$ in terms of generators and relations. Moreover we determine an upper bound for its running time using that $\Lambda^*\backslash\CT$ is {\em almost} Ramanujan.

\bigskip

\section{Introduction}
\label{Sec-Intro}

A major recent theme in explicit arithmetic geometry over $\BQ$ or more general number fields has been the development and implementation of algorithms to compute automorphic forms \cite{Cremona,Dembele,GreenbergVoight,GunnelsYasaki,Stein}. More precisely, these algorithms compute the Hecke action on spaces of modular forms for a given level and weight. Typically these algorithms proceed in three steps: (i) a combinatorial or geometric model is provided in which one can compute the Hecke action; (ii) on the model one performs some precomputations such as the computations of ideal classes of a maximal order in a quaternion division algebra, or the computation of a fundamental domain; (iii) on the data provided by~(ii) one implements the Hecke action.

The present article is concerned with an analogous algorithm over function fields whose ultimate goal is the computation of Drinfeld modular forms as well as automorphic forms. For $\GL_2$ over function fields, (i) and (ii) were solved in \cite{Teitelbaum} and \cite{Teitelbaum2,GekelerNonnengard}, respectively. Here we will be concerned with inner forms of $\GL_2$ that correspond to the unit group of a quaternion division algebra split at $\infty$. In this setting an extension of \cite{Teitelbaum} is part of \cite{ButenuthThesis}. The sought-for combinatorial description of the forms to be computed is given in terms of harmonic cocycles on  the Bruhat-Tits tree which are invariant under the action of an arithmetic subgroup $\Gamma$ defined from the division algebra. The main precomputation that makes up step (ii) is that of a fundamental domain of $\CT$ under the action of $\Gamma$. This can be thought of as an analog of \cite{Voight}. Due to the different underlying {\em geometry} the methods employed are completely different. 

To describe the output of our algorithm and some consequences note first that in our setting of a quaternion division algebra split at $\infty$, the quotient $\Gamma\backslash\CT$ is a finite graph. The fundamental domain with an edge pairing 
that we compute consists of the following data: 
\begin{enumerate}
\item  a finite subtree $Y\subset \CT$ whose image $\overline Y$ in $\Gamma\backslash\CT$ is a maximal spanning tree, i.e., $\overline Y$ is a tree such that adding any edge of $\Gamma\backslash\CT$ to it will create a cycle. 
\item \label{GluedEdges}for any edge $\bar e$ of $\Gamma\backslash\CT\setminus \overline Y$, an edge $e$ of $\CT$ connected to $Y$ that maps to $\bar e$ and a the gluing datum that connects the loose vertex of this edge via the action of $\Gamma$ to a vertex of~$Y$.
\end{enumerate}
Compared to output of \cite{Voight}, what we compute is the analog of fundamental domain together with a side pairing. As explained in \cite[\S~I.4]{Serre1}, this data yields a presentation of the group $\Gamma$ in terms of explicit generators and relations. Moreover the fundamental domain data computed provides an efficient reduction algorithm on the tree: to any edge it computes its representative in $Y'$, by which we mean the union of $Y$ with the edges in~(\ref{GluedEdges}). Reinterpreted in terms of group theory, a fundamental domain with an edge pairing yields an efficient algorithm to solve the word problem for~$\Gamma$.

Observing that a finite cover of $\Gamma\backslash\CT$ is a Ramanujan graph, yields a bound on the diameter of $\Gamma\backslash\CT$. This in turn we use to bound the complexity of our algorithm, to bound the size of $Y'$, and to bound the size of the representatives of $Y'$ in terms of a natural height on the $2\times 2$-matrices over the function field completed at $\infty$. Our main result is therefore the existence of an effective algorithm together with precise complexity bounds. An implementation can be obtained on request from the second author. 

A theoretical result on the size of a minimal generating set for $\Gamma$ and the (logarithmic) height of its generators was obtained by different methods in \cite{Papikian2}. The bound on the height here is better by a factor of $2$. The results in \cite{Papikian2} also prove the existence of an algorithm to compute a fundamental domain. An implementation or a detailed analysis of that algorithm have not yet been carried out. Both \cite{Papikian2} and the present article rely heavily on \cite{Papikian} by Papikian.

\smallskip

We conclude by a short overview of the article: in Sections~\ref{Sec-GraphTheory}, \ref{Sec-BT-Tree} and \ref{Sec-QuatAlg}, we recall basic notions and results from graph theory, on the Bruhat-Tits tree and from the theory of quaternion algebras. Section~\ref{Sec-QuatGraph} introduces the main object of this article, the action of $\Gamma$ on the Bruhat-Tits tree $\CT$, and states relevant results on the resulting quaternion graph $\Gamma\backslash\CT$. Our basic algorithm is presented in Section~\ref{Sec-Algo}, except that we use one unproved subroutine which is the content of the following Section~\ref{SecCompHom}. The final Sections~\ref{Sec-Presentations} and~\ref{Sec-Complex} present on the one hand the applications of the algorithm to the presentation of $\Gamma$ in terms of generators and relations and to the word problem, and on the other hand the complexity analysis of the algorithm based on the fact that $\Gamma\backslash\CT$ has a finite cover that is Ramanujan.

\medskip

{\bf Acknowledgments:} For several useful discussions we wish to thank Mihran Papikian and John Voight. We also want to heartily thank the anonymous referee whose many comments and suggestions greatly improved the readability of the present article.  During this work, the authors were supported by the Sonderforschungsbereich/Transregio 45 {\em Periods, Moduli Spaces and Arithmetic of Algebraic Varieties} and by the  DFG priority project SPP 1489. The implementation of the algorithm is based on the computer algebra system Magma \cite{Bosma}. 


\subsubsection*{Notation}
Throughout this article $K = \BFq(T)$ will denote the rational function field over $\BFq$. As usual, the infinite valuation $v_\infty$ on $K$ is defined by $v_\infty(\frac{f}{g}) = \deg(g) - \deg(f)$ for $f, g \in \Aa=\BF_q[T], g \neq 0$ and $v_\infty(0) = \infty$. 
Then $\pi=1/T$ is a uniformizer for $v_\infty$, the corresponding completion of $K$ is $K_\infty = \BFq((\pi))$ and we write $\CO_\infty$ for its ring of integers.

\begin{Rem}
The restriction that $q$ be odd is for the sake of simplicity of exposition. To treat the case that $q$ is even, one needs to redo all of Section~\ref{Sec-QuatAlg} with very little overlap with the odd case. Some changes are also necessary in Section~\ref{SecCompHom}. All other results hold independently of $q$ being even or odd. We have implemented our algorithm also for even $q$. Theoretical discussions on $q$ even are sketched in~\cite{ButenuthThesis}.
\end{Rem}

\section{Notions from graph theory}
\label{Sec-GraphTheory}

\begin{Def}\label{DefDirMulGraph}
\begin{enumerate}
\item A {\em (directed multi-)graph} $\MCG$ is a pair $(\Ver(\MCG), \Edg(\MCG))$ where $\Ver(\MCG)$ is a (possibly infinite) set and $\Edg(\MCG)$ is a subset of $\Ver(\CG)\times\Ver(\CG)\times\BZ_{\ge0}$ such that 
\begin{enumerate}
\item if $e = (v, v^\prime, i)$ lies in $\Edg(\MCG)$, then so does its opposite $e^\star = (v^\prime, v, i)$,
\item for any $(v,v')\in \Ver(\CG)\times\Ver(\CG)$, the set $\{i\in\BZ_{\ge0}\mid (v,v^\prime,i)\in\Edg(\CG)\}$ is a finite initial segment of $\BZ_{\ge0}$ of cardinality denoted by $n_{v,v'}$,
\item for any $v \in \Ver(\CG)$, the set $\Nbs(v) := \{ v^\prime \in \Ver(\MCG) \mid (v, v^\prime, 0) \in \Edg(\MCG) \}$ is finite.
\end{enumerate}
\end{enumerate}
\end{Def}

An element $v \in \Ver(\MCG)$ is called a {\em vertex}, an element $e \in \Edg(\MCG)$ is called an {\em (oriented) edge} and an element in $\Ver(\MCG) \sqcup \Edg(\MCG)$ is called a {\em simplex}. For each edge $e = (v, v^\prime, i) \in \Edg(\MCG)$ we call $o(e) := v$ the {\em origin} of $e$ and $t(e) := v^\prime$ the {\em target} of $e$. If there is only one edge between vertices $v$, $v^\prime$ of $\MCG$
we simply write $(v, v')$ instead of $(v, v', 0)$. Two vertices $v, v^\prime$ are called {\em adjacent}, if $\lbrace v, v^\prime \rbrace = \lbrace o(e), t(e) \rbrace$ for some edge~$e$.

\begin{Def}
An edge $e$ with $o(e) = t(e)$ is called a {\em loop}. A vertex $v$ with $\#\Nbs(v)=1$ is called {\em terminal}.
\end{Def}

Let $v, v^\prime \in \Ver(\MCG)$. A {\em path} from $v$ to $v^\prime$ is a finite subset $\lbrace e_1, \dots, e_k \rbrace$ of $\Edg(\MCG)$ such that $t(e_i) = o(e_{i + 1})$ for all $i = 1, \dots, k - 1$ and $o(e_1) = v, t(e_k) = v^\prime$. The integer $k$ is called the {\em length} of the path $\lbrace e_1, \dots, e_k \rbrace$. The {\em distance} from $v$ to $v^\prime$, denoted $d(v, v^\prime)$, is the minimal length among all paths from $v$ to $v^\prime$ (or $\infty$ if no such path exists). A path $\lbrace e_1, \dots, e_k \rbrace$ from $v$ to  $v^\prime$ without backtracking, i.e., such that for no $i$ we have $e_i=e_{i-1}^\star$, is called a {\em geodesic}. Note that the length of a geodesic need not be $d(v,v')$ but that $d(v,v')$ is attained for a geodesic.

A graph $\MCG$ is {\em connected} if for any two vertices $v, v^\prime \in \Ver(\MCG)$ there is a path from $v$ to $v^\prime$. A {\em cycle} of $\MCG$ is a geodesic from some vertex $v$ to itself. Therefore a loop is a cycle of length one. A graph $\MCG$ is {\em cycle-free} if it contains no cycles. A {\em tree} is a connected, cycle-free graph. If $\MCG$ is a tree, then any two vertices of $\CG$ are connected by a unique geodesic. 

A {\em subgraph} $\MCG^\prime \subseteq \MCG$ is a graph $\CG'$ such that $\Ver(\MCG^\prime) \subseteq \Ver(\MCG)$ and $\Edg(\MCG^\prime) \subseteq \Edg(\MCG)$. Any subgraph $\MCS \subseteq \MCG$ which is a tree is called {\em subtree}. A {\em maximal subtree} is a subtree which is maximal under inclusion among all subtrees of~$\CG$.

The {\em degree} of $v \in \Ver(\MCG)$ is 
$$\degree(v):=\degree_\CG(v) := \sum_{e\in\Edg(\CG):o(e)=v} n_{v,t(e)}.$$
Thus $v$ is terminal precisely if $\degree(v)=1$. A graph $\MCG$ is called {\em $k$-regular} if for all vertices $v \in \Ver(\MCG)$ we have $\degree(v) = k$. 

A graph $\MCG$ is {\em finite}, if $\# \Ver(\MCG) < \infty$. Then also $\#\Edg(\CG)<\infty$, since $\degree(v)$ is finite for all $v\in\Ver(\CG)$. The {\em diameter} of a (finite) graph $\MCG$ is 
$$\diam(\MCG) := \max_{v, v^\prime \in \Ver(G)} d(v, v^\prime).$$

\begin{Def}
The {\em first Betti number $h_1(\MCG)$} of a finite connected graph is
	$$h_1(\MCG) := \frac{\#\Edg(\MCG)}{2} - \#\Ver(\MCG) + 1.$$
\end{Def}

A graph $\MCG$ defines an abstract simplicial set. Its geometric realization is a topological space $|\MCG|$. For finite graphs one has $h_1(\MCG)=\dim_\BQ H_1(|\MCG|, \BQ)$, i.e., the Betti number counts the number of independent cycles of~$\MCG$.

\section{The Bruhat-Tits tree}
\label{Sec-BT-Tree}

In this section, we recall the definition of the Bruhat-Tits tree for the group $\PGL_2(K_\infty)$. It is an important combinatorial object for the arithmetic of $K$. The material can be found in \cite{Serre1}.

\smallskip

One defines a graph $(\Ver(\CT),\Edg(\CT))$ as follows: Two $\CO_\infty$-lattices $L,L'\subset K_\infty^2$ are called equivalent if there is a $\lambda \in K_\infty^*$ with $L'=\lambda L$. The set $\Ver(\CT)$ is the set of equivalence classes $[L]$ of such lattices. The set $\Edg(\CT)$ is the set of triples $([L],[L'])$ such that $L,L'\subset K_\infty^*$ are $\CO_\infty$-lattices with $\pi L\subsetneq L'\subsetneq L$, or equivalently such that $L'\subsetneq L$ and $L/L'\cong\BF_q$ as $\CO_\infty$-modules. In particular there is at most one edge between any two vertices. 

By \cite[\S~II.1]{Serre1} the graph $\CT=(\Ver(\CT),\Edg(\CT))$ is a $q+1$-regular tree -- recall that $q$ is the cardinality of the residue field of $K_\infty$. The group $\PGL_2(K_\infty)$ acts naturally on lattice classes by left multiplication $(g,[L])\mapsto [gL]$. This induces an action on $\CT$. Because the definitions are a special case of a general construction, one calls $\CT$ the {\em Bruhat-Tits tree for $\PGL_2(K_\infty)$}.

Let $e_1 = (1, 0)^t$ and $e_2 = (0, 1)^t$ be the standard basis of $K_\infty^2$, thought of as column vectors. Write $\CO_\infty^2$ for $\CO_\infty e_1\oplus\CO_\infty e_2$. The following result is well-known and straightforward, using the class equation and the transitive action of $\GL_2(K_\infty)$ on bases of $K_\infty^2$.
\begin{Prop}\label{PropBTTreeVertices}
The map 
\begin{eqnarray}
\phi: \GL_2(K_\infty) / \GL_2(\CO_\infty) K_\infty^* &\to& \Ver(\MCT) \nonumber \\
 A &\mapsto& \left[ A \CO_\infty^2 \right]   \nonumber
\end{eqnarray}
is a bijection of left $\GL_2(K_\infty)$-sets.
\end{Prop}

The map $\phi$ of the above proposition allows us to represent vertices of $\CT$ by elements of $\GL_2(K_\infty)$. Row-reduction to the echelon form of a matrix yields a standard representative in $\GL_2(K_\infty)$ as expressed by the following result.
\begin{Lem}\label{LemVertexNF}
Every class of $\GL_2(K_\infty) / \GL_2(\CO_\infty) K_\infty^*$ has a unique representative of the
form $$\begin{pmatrix} \pi^n & g \\ 0 & 1 \end{pmatrix}$$
with $n \in \BZ$ and $g \in K_\infty / \pi^n \CO_\infty$,  called its \em{vertex normal form}.
\end{Lem}

We also need a criterion for adjacency for matrices in vertex normal form.
\begin{Lem}\label{LemBTTreeVertMatr}
Consider the two matrices in vertex normal form
$$A := \begin{pmatrix} \pi^n & g \\ 0 & 1 \end{pmatrix}, 
B := \begin{pmatrix} \pi^{n + 1} & g + \alpha \pi^n \\ 0 & 1 \end{pmatrix}$$
with $n \in \BZ, \alpha \in \BFq, g \in K_\infty / \pi^n \CO_\infty$ and let $L_1$ and $L_2$ be the 
two lattices $$L_1 :=A\CO_\infty^2, L_2 := B\CO_\infty^2.$$
Then $L_1 \supset L_2$ and $L_1 / L_2 \cong \BFq$.
\end{Lem}
\begin{Rem}
Lemma~\ref{LemBTTreeVertMatr} only displays $q$ vertices adjacent to $[L_1]$. The missing one is the class of $\begin{pmatrix} \pi^{n-1} & g \\ 0 & 1 \end{pmatrix}\CO_\infty^2.$ with $g$ now being replaced its class in $K_\infty / \pi^{n-1} \CO_\infty$. 
\end{Rem}

Figure~\ref{fig:BTTreeVertMatr} below illustrates the tree together with the matrices in normal form
corresponding to vertices. The identification is clear from the previous lemma. 
Note that each line in the picture symbolizes a whole fan expanding to the 
right. The elements $\alpha \in \BF_q^*$, $\beta \in \BFq$ agree on each fan.

\setlength{\unitlength}{0.75cm}

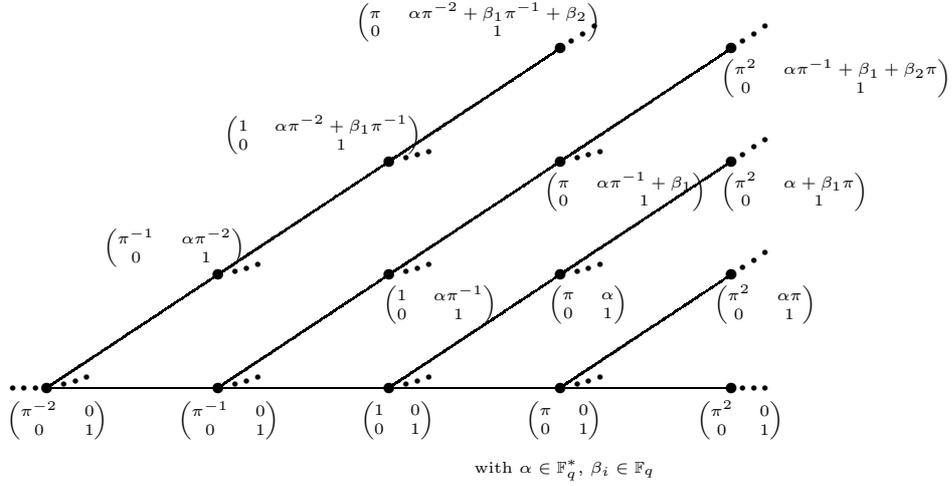
\begin{figure}[ht]
\centering
\begin{picture}(15,9)
\put(0.8,2.06){\circle*{0.08}}
\put(1,2.12){\circle*{0.08}}
\put(1.2,2.18){\circle*{0.08}}

\put(3.8,2.06){\circle*{0.08}}
\put(4,2.12){\circle*{0.08}}
\put(4.2,2.18){\circle*{0.08}}

\put(6.8,2.06){\circle*{0.08}}
\put(7,2.12){\circle*{0.08}}
\put(7.2,2.18){\circle*{0.08}}

\put(9.8,2.06){\circle*{0.08}}
\put(10,2.12){\circle*{0.08}}
\put(10.2,2.18){\circle*{0.08}}

\put(3.8,4.06){\circle*{0.08}}
\put(4,4.12){\circle*{0.08}}
\put(4.2,4.18){\circle*{0.08}}

\put(6.8,4.06){\circle*{0.08}}
\put(7,4.12){\circle*{0.08}}
\put(7.2,4.18){\circle*{0.08}}

\put(9.8,4.06){\circle*{0.08}}
\put(10,4.12){\circle*{0.08}}
\put(10.2,4.18){\circle*{0.08}}

\put(6.8,6.06){\circle*{0.08}}
\put(7,6.12){\circle*{0.08}}
\put(7.2,6.18){\circle*{0.08}}

\put(9.8,6.06){\circle*{0.08}}
\put(10,6.12){\circle*{0.08}}
\put(10.2,6.18){\circle*{0.08}}

\put(0.5,2){\circle*{0.2}}
\put(3.5,2){\circle*{0.2}}
\put(6.5,2){\circle*{0.2}}
\put(9.5,2){\circle*{0.2}}
\put(12.5,2){\circle*{0.2}}
\put(12.7,2){\circle*{0.08}}
\put(13.1,2){\circle*{0.08}}
\put(12.9,2){\circle*{0.08}}

\put(3.5,4){\circle*{0.2}}
\put(6.5,4){\circle*{0.2}}
\put(9.5,4){\circle*{0.2}}
\put(12.5,4){\circle*{0.2}}

\qbezier(0.5,2)(0.5,2)(3.5,4)
\qbezier(3.5,2)(3.5,2)(6.5,4)
\qbezier(6.5,2)(6.5,2)(9.5,4)
\qbezier(9.5,2)(9.5,2)(12.5,4)

\qbezier(3.5,4)(3.5,4)(6.5,6)
\qbezier(6.5,4)(6.5,4)(9.5,6)
\qbezier(9.5,4)(9.5,4)(12.5,6)

\qbezier(6.5,6)(6.5,6)(9.5,8)
\qbezier(9.5,6)(9.5,6)(12.5,8)


\put(6.5,6){\circle*{0.2}}
\put(9.5,6){\circle*{0.2}}
\put(12.5,6){\circle*{0.2}}

\put(9.5,8){\circle*{0.2}}
\put(12.5,8){\circle*{0.2}}

\put(12.7,8.13){\circle*{0.08}}
\put(12.9,8.26){\circle*{0.08}}
\put(13.1,8.39){\circle*{0.08}}

\put(12.7,6.13){\circle*{0.08}}
\put(12.9,6.26){\circle*{0.08}}
\put(13.1,6.39){\circle*{0.08}}

\put(12.7,4.13){\circle*{0.08}}
\put(12.9,4.26){\circle*{0.08}}
\put(13.1,4.39){\circle*{0.08}}

\put(9.7,8.13){\circle*{0.08}}
\put(9.9,8.26){\circle*{0.08}}
\put(10.1,8.39){\circle*{0.08}}



\put(0.3,2){\circle*{0.08}}
\put(0.1,2){\circle*{0.08}}
\put(-0.1,2){\circle*{0.08}}

\qbezier(0.5,2)(0.5,2)(3.5,2)
\qbezier(3.5,2)(3.5,2)(6.5,2)
\qbezier(6.5,2)(6.5,2)(9.5,2)
\qbezier(9.5,2)(9.5,2)(12.5,2)
\put(-0.2,1.35){\begin{tiny}$\begin{pmatrix} \pi^{-2} & 0 \\ 0 & 1 \end{pmatrix}$\end{tiny}}
\put(2.8,1.35){\begin{tiny}$\begin{pmatrix} \pi^{-1} & 0 \\ 0 & 1 \end{pmatrix}$\end{tiny}}
\put(6,1.35){\begin{tiny}$\begin{pmatrix} 1 & 0 \\ 0 & 1 \end{pmatrix}$\end{tiny}}
\put(8.9,1.35){\begin{tiny}$\begin{pmatrix} \pi & 0 \\ 0 & 1 \end{pmatrix}$\end{tiny}}
\put(11.9,1.35){\begin{tiny}$\begin{pmatrix} \pi^2 & 0 \\ 0 & 1 \end{pmatrix}$\end{tiny}}
\put(12.2,5.4){\begin{tiny}
$\begin{pmatrix} \pi^{2} & \alpha + \beta_1 \pi \\ 0 & 1 \end{pmatrix}$\end{tiny}}
\put(12.2,7.4){\begin{tiny}
$\begin{pmatrix} \pi^{2} & \alpha \pi^{-1} + \beta_1 + \beta_2 \pi \\ 0 & 1 \end{pmatrix}$\end{tiny}}

\put(9.1,5.4){\begin{tiny}
$\begin{pmatrix} \pi & \alpha \pi^{-1} + \beta_1 \\ 0 & 1 \end{pmatrix}$\end{tiny}}
\put(3.5,6.4){\begin{tiny}
$\begin{pmatrix} 1 & \alpha \pi^{-2} + \beta_1 \pi^{-1} \\ 0 & 1 \end{pmatrix}$\end{tiny}}
\put(5.8,8.4){\begin{tiny}
$\begin{pmatrix} \pi & \alpha \pi^{-2} + \beta_1 \pi^{-1} + \beta_2 \\ 0 & 1 \end{pmatrix}$\end{tiny}}

\put(12.2,3.4){\begin{tiny}$\begin{pmatrix} \pi^{2} & \alpha \pi \\ 0 & 1 \end{pmatrix}$\end{tiny}}
\put(9.3,3.4){\begin{tiny}$\begin{pmatrix} \pi & \alpha \\ 0 & 1 \end{pmatrix}$\end{tiny}}
\put(6.4,3.4){\begin{tiny}$\begin{pmatrix} 1 & \alpha \pi^{-1} \\ 0 & 1 \end{pmatrix}$\end{tiny}}
\put(1.5,4.4){\begin{tiny}$\begin{pmatrix} \pi^{-1} & \alpha \pi^{-2} \\ 0 & 1 \end{pmatrix}$\end{tiny}}

\put(8,0.5){\begin{tiny}with $\alpha \in \BF_q^*$, $\beta_i \in \BFq$ \end{tiny}} 
\end{picture}
\caption{The tree $\MCT$ with the corresponding matrices}
\label{fig:BTTreeVertMatr}

\end{figure}

Write $L(n, g)$ for the $\CO_\infty$-lattice $\langle v_1, v_2 \rangle_{\CO_\infty}$ 
where $v_1 = \begin{pmatrix} \pi^n \\ 0 \end{pmatrix}$ and $v_2 = \begin{pmatrix} g \\ 1 \end{pmatrix}$. 
Note that $L(n, g) = L(n, g^\prime)$ if and only if $g \equiv g^\prime \pmod{ \pi^n \CO_\infty}$.

\begin{Rem}\label{RemGeodesivToStandVertex}
 For $n \in \BZ$, $g \in K_\infty$ we define $$\delta:=\deg_n(g) := \min \lbrace i \in \BZ_{\ge0} \mid g \in \pi^{n - i} \CO_\infty \rbrace.$$
Then the path from $L(n, g)$ to $L(0,0)$ in $\MCT$ is 
%
%
%
 $$L(n, g)  \text{ ----- } L(n - 1, g) \text{ ----- } \dots  \text{ ----- }
L(n - \delta, g)= L(n - \delta, 0)\text{ --- }  $$ 
$$\text{ ----- } L(\sgn(n - \delta)\cdot (| n - \delta | - 1), 0)
\text{ ----- }\dots \text{ ----- } L(\sgn(n - \delta)\cdot 1, 0) \text{ ----- } L(0, 0)$$
In particular the distance between $L(n, g)$ and $L(0,0)$ is $\deg_n(g) + | n - \deg_n(g) |$.
\end{Rem}

\section{Quaternion algebras}
\label{Sec-QuatAlg}

We recall standard facts on quaternion algebras over $K=\BF_q(T)$ and over completions of $K$, and on orders over $\Aa=\BF_q[T]$.  We assume throughout that $q$ is odd. Our basic references are \cite[Kap.~IX]{Jantzen-Schwermer} and \cite{Vigneras}. Many results stated are true more generally. However, we confine ourselves to the case at hand.

\medskip

A quaternion algebra over a field $F$ is a central simple algebra of dimension $4$ over $F$. It is either  isomorphic to $M_2(F)$ or 
a division algebra. One has the following well-known construction of quaternion algebras.

\begin{Cons}\label{ConsQuatAlg}
For $a, b \in F^*$ one defines $\QuatAlg{a}{b}{F}$ as the $K$-algebra with $F$-basis $1, i, j, ij$ and relations
$i^2 = a, j^2 = b, ij=-ji$.
\end{Cons}
The relations can be expanded to a $4\times 4$ multiplication table for the given $F$-basis of $\QuatAlg{a}{b}{F}$. One shows that $\QuatAlg{a}{b}{F}$ defines a quaternion algebra over $F$, and that conversely any quaternion algebra over $F$ can be obtained via this construction for a suitable choice of $a,b\in F^*$.

Among other things, a quaternion algebra $D$ over $F$ carries a {\em reduced norm map} $\nrd\colon D\to F$ which defines a quadratic form on $D$. For $D= \QuatAlg{a}{b}{F}$ the reduced norm has the explicit expression
$$\nrd(\gamma) \ = \ \gamma\cdot\bar\gamma\,\,\,=\,\,\,\gamma_1^2-a^2\gamma_2^2-b^2\gamma_3^2+ab\gamma_4^2$$
for any $\gamma=\gamma_1+\gamma_2i+\gamma_3j+\gamma_4ij\in D$ with $(\gamma_1,\gamma_2,\gamma_3,\gamma_4)\in F^4$. If $D$ is isomorphic to $M_2(F)$, then $\nrd$ is simple the determinant map $\det\colon M_2(F)\to F$.

\medskip

Let now $D$ denote a quaternion algebra over $K$. Then $D_\Fp:=D\otimes_KK_\Fp$ is a quaternion algebra over the completion $K_\Fp$  for any place $\Fp$ of~$K$.

\begin{Def}\label{DefRamification}
$D$ is {\em ramified} at $\Fp$ if and only if $D_\Fp$ is a division algebra.  
\end{Def}

\begin{Def}\label{DefHilbertSymbol} The {\em Hilbert symbol} of a pair $(a, b) \in K^2$ at a place $\Fp$ is
$$(a,b)_{K_{\Fp}} := \begin{cases}+1 & \QuatAlg{a}{b}{K} \text{ is unramified at } 
\Fp\\ -1 & \QuatAlg{a}{b}{K} \text{ is ramified at } \Fp. \end{cases}$$
\end{Def}

\begin{Def}\label{DefLegrendeSymbol}
For $a\in{\Aa}$ and $\varpi$ an irreducible element of ${\Aa}$, the {\em Legendre symbol of $a$ at $\varpi$} is
$$\Leg{a}{\varpi} := \begin{cases} \phantom{-}1 & a \notin \varpi \Aa \text{ and } a \text{ is a square modulo } \varpi \\
				   -1 & a \text{ is a non-square modulo } \varpi \\
				     \phantom{-}0 & a\in \varpi\Aa.
                     \end{cases}
$$
\end{Def}

By adaptating to the function-field situation the proof of \cite[Ch.~III, Thm.~1]{Serre2}, the following result is straightforward.
\begin{Prop}\label{PropComputationHilbertSymbol}
Suppose $q$ is odd. Write $\Fp = (\varpi)$ and let $a = \varpi^{\alpha} u, b = \varpi^{\beta} v$ 
with $u, v \in O_{K_{\Fp}}^*, \alpha, \beta \in \BZ$ 
and let $\epsilon(\Fp) := \frac{q - 1}{2} \deg(\varpi) \pmod{2}$. Then
$$(a, b)_{K_\Fp} = (-1)^{\alpha \beta \epsilon(\Fp)} \Leg{u}{\varpi}^\beta \Leg{v}{\varpi}^\alpha.$$ 
\end{Prop}

Let $R$ denote the set of all ramified places of~$D$. Then \cite[Lem.~III.3.1 and Thme.~III.3.1]{Vigneras} yields the following.
\begin{Prop}\label{PropRamifiedFiniteEven}
The cardinality of $R$ is finite and even and $D$ is up to isomorphism uniquely determined by $R$. The set $R$ is empty if and only if $D\cong M_2(K)$.
\end{Prop}

The ideal $\Fr := \prod_{\Fp \in R} \Fp$ of ${\Aa}$ is called the {\em discriminant} of $D$. Let $r\in{\Aa}$ be the monic generator of $\Fr$.

\begin{Ass}\label{MainAss} For the remainder of this article, we assume that $D$ is a division quaternion algebra which is unramified at $\infty$, , i.e., that $D$ is an indefinite quaternion algebra over ${\Aa}$. We also fix an isomorphism $D_\infty\cong M_2(K_\infty)$.
\end{Ass}

Let $\Lambda$ be an order of $D$ over ${\Aa}$. It is free over ${\Aa}$ of rank $4$ and so we may choose a basis $f_1,\ldots,f_4$. The ideal generated by 
$$\disc(f_1, \dots, f_4) := \det(\trd(f_i f_j))_{i, j = 1, \dots, 4}$$
is independent of the chosen basis. By \cite[Lem.~I.4.7]{Vigneras}, this ideal is a square and one defines the {\em reduced discriminant} $\rdisc(\Lambda)$ of $\Lambda$ as the square root of this ideal. One deduces a criterion for an order to be maximal, see \cite[Cor.~III.5.3]{Vigneras}.
\begin{Prop}\label{DiscCritForMaxOrder}
An ${\Aa}$-order $\Lambda$ is maximal in $D$ if and only if $\rdisc(\Lambda)=\Fr$.
\end{Prop}
Since $D$ is split at infinity and $K$ has class number $1$, \cite[Cor.~III.5.7]{Vigneras} yields:
\begin{Prop}
All maximal ${\Aa}$-orders $\Lambda$ in~$D$ are conjugate under $D^*$.
\end{Prop}

Let $\Gamma := \Lambda^*$ be the group of units of of a maximal order $\Lambda$. By what we have said so far, $\Gamma$ depends uniquely up to conjugation on $D$, i.e., on $K$ and~$R$. 

\smallskip

From global to local compatibilities and explicit local results, one deduces the following assertions.
\begin{Lem}\label{EmbeddingAbstract}
\begin{enumerate}
\item The reduced norm $\nrd$ maps $\Lambda$ to $\Aa$.
\item $\Gamma=\{\gamma\in\Lambda\mid \nrd(\gamma)\in\BF_q^*\}.$
\item The embedding $\iota\colon\! D\into D_\infty\cong M_2(K_\infty)$ restricts to a group monomorphism
$$\Gamma \longinto  \SL_2(K_\infty) \begin{pmatrix} \BF_q^* & 0 \\ 0 & 1 \end{pmatrix} \subset \GL_2(K_\infty).$$
\end{enumerate}
\end{Lem}

The following result is well-known. In lack of an explicit reference, we shall give a proof.
\begin{Prop}\label{PropImageDiscrete}
Via $\iota$ the group $\Gamma$ is a discrete subgroup of $\GL_2(K_\infty)$.
\end{Prop}

\begin{proof}
The open sets $\lbrace 1 + \pi^n M_2(\CO_\infty) \mid n \in \BN \rbrace$ form a basis of open neighborhoods of $1$ in $\GL_2(K_\infty)$. After shifting by $1$ it suffices to show that $\Lambda \cap M_2(\CO_\infty)$ is finite, or in other words that $\Lambda$ is discrete in $M_2(K_\infty)$.

To see the discreteness, let $\MCD$ be the unique locally free coherent sheaf of rings of rank $4$ over $\BP^1_\BFq$ such that $\Lambda \cong \Gamma(\BA^1_\BFq, \MCD)$ and such that the completed stalk at infinity satisfies $\MCD_\infty \cong M_2(\CO_\infty)$. Then $\Lambda \cap M_2(\CO_\infty) = H^0(\BP^1_\BFq, \MCD)$. By the Riemann-Roch Theorem, this is a finite-dimensional $\BFq$-vector space. 

Alternatively, for $D$ and $\Lambda$ constructed later in Propositions~\ref{ExplicitD} and \ref{TheMaxOrderLambda}, and the embedding from Lemma~\ref{LemEmbedding}, the discreteness can be verified explicitly, by proving that an $A$-Basis of $\Lambda$ maps to a $K_\infty$-basis of $D_\infty$.
\end{proof}

\medskip

Given an even set $R$ of finite places of $K$ at which $D$ is ramified, the algorithm described in Sections~\ref{Sec-Algo} and~\ref{SecCompHom} will be based on a concrete model for $(D,\Lambda)$. In the remainder of this section we describe such a model. It will consist of an explicit pair $(a,b)\in K^*$ such that $D\cong\QuatAlg{a}{b}{K}$ and an explicit basis of a maximal ${\Aa}$-order $\Lambda$ of $\QuatAlg{a}{b}{K}$.

Let $l\ge2$ be even and let $R$ be a set of $l$-distinct prime ideals $\{ \Fp_1, \dots, \Fp_l \}$ of ${\Aa}$. Denote by $\varpi_i$ the unique monic (irreducible) generator of $\Fp_i$. Set $r:=\prod_i\varpi_i$ and $\Fr:=\prod_i\Fp_i$ where the index $i$ ranges over $1,\ldots,l$.

\begin{Lem}\label{LemAlphaExists}
 There is an irreducible monic polynomial $\alpha \in {\Aa}$ of even degree such that 
\begin{equation}\label{CondOnAlpha}
\Leg{\alpha}{\varpi_i} = -1 \text{ for all }i.
\end{equation}
Any such $\alpha$ also satisfies $\Leg{r}{\alpha} = 1.$
\end{Lem}
\begin{proof} Choose any $a \in {\Aa}$ such that $$\Leg{a}{\varpi_i} = -1$$ for all $i$. This can be done using the
Chinese remainder theorem. By the strong form of the function field analogue of Dirichlet's theorem on primes in arithmetic progression, \cite[Thm.~4.8]{Rosen}, the set $\lbrace a + r b \mid b \in {\Aa} \rbrace$ contains an irreducible monic 
polynomial $\alpha$ of even degree. Since $\alpha \equiv a \pmod {\varpi_i}$ we have 
$$\Leg{\alpha}{\varpi_i} = -1$$ for all $i$.
By quadratic reciprocity, \cite[Thm.~3.3]{Rosen}, we deduce
$$\Leg{\varpi_i}{\alpha} = (-1)^{\frac{q - 1}{2} \deg{\alpha}\deg{\varpi_i}} \Leg{\alpha}{\varpi_i} = -1$$
since $\deg(\alpha)$ is even. But then because $l$ is even, we find
$$\Leg{r}{\alpha} = \prod_{i=1}^l \Leg{\varpi_i}{\alpha} = (-1)^l = 1.$$
\end{proof}

\begin{Rem}
In practice $\alpha$ is rapidly found by the following simple search:
\begin{itemize}\advance\labelsep by -1.5em
\item[Step 1:] \hspace*{1.5em}Start with $m = 2$.
\item[Step 2:] \hspace*{1.5em}Check for all monic irreducible $\alpha\in{\Aa}$ of degree $m$ whether 
$\Leg{\alpha}{\varpi_i} = -1$ for all $1 \leq i \leq l$. 
\item[Step 3:] \hspace*{1.5em}If we found an $\alpha$ then stop. Else increase $m$ by $2$ and go back to Step~2.
\end{itemize}
\end{Rem}

In the function field setting \cite{MurtyScherk} gives an unconditional effective version of the \v{C}ebotarov density theorem. This allows us to make Lemma~\ref{LemAlphaExists} effective, i.e., to give explicit bounds on $\deg(\alpha)$ in terms of $\deg(r)$. 
\begin{Prop}\label{PropChebtor}
Abbreviate $d:=\deg(r)$. The following table gives upper bounds on $d_\alpha:=\deg(\alpha)$ depending on $q$ and $l$:
\begin{center}
\begin{tabular}{c|c|c|c|c|c|c|c|c|c|}
&\multicolumn{3}{|c|}{$q=3$}&\multicolumn{2}{c|}{$q=5,7$}&\multicolumn{2}{c|}{$q=9$}&\multicolumn{2}{c|}{$q\ge11$}\\
\cline{2-10}
&$l\le4$&$l=6$&$8\leq l$&$l\le6$&$8\le l$&$l\le4$&$6\le l$&$l=2$&$4\le l$\\
\hline
$d_{\alpha}\le$&$ d+7$&$d+5$&$d+1$&$d+3$&$d+1$&$d+3$&$d+1$&$d+3$&$d+1$\\
\end{tabular}\end{center}
\end{Prop}
A basic reference for the results on function fields used in the following proof is~\cite{Stichtenoth}.
\begin{proof} Let $K':=K(\sqrt{\varpi_1},\ldots,\sqrt{\varpi_l})$. Then $K'/K$ is a Galois extension with Galois group isomorphic to $\{\pm1\}^l$ with $\{\pm1\}\cong\BZ/(2)$; it is branch locus in $K$ is the divisor $D$ consisting of the sum of the $(\varpi_k)$ and (possibly) $\infty$; the constant field of $K'$ is again $\BFq$. Denote by $g'$ the genus of $K'$ and by $D'$ the ramification divisor of $K'/K$. The ramification degree at all places is $1$ or $2$ and hence tame because $q$ is odd. It follows that $\deg(D')=\#G/2\cdot \deg(D)$. 

Let $\pi(k)$ denote the places of $K$ of degree $k$; let $\pi_C(k)$ denote the places $\Fp$ of $K$ of degree $k$ for which $\Frob_\Fp=(-1,\ldots,-1)\in\{\pm1\}^l$. Note that the elements of $\pi_C(k)$ are in bijection to the monic irreducible polynomials $\alpha$ of degree $k$ which satisfy the conditions (\ref{CondOnAlpha}). The following two inequalities are from \cite[Thm.~1 and (1.1)]{MurtyScherk} and the Hurwitz formula, respectively:
\begin{equation}\label{CebCond1}
\big|  \pi_C(k) - \frac1{\#G} \pi(k)\big| \le 2g'\frac1{\#G} \frac{q^{k/2}}k + 2 \frac{q^{k/2}}k +\big(1+\frac1k\big)\deg(D').
\end{equation}
\begin{equation}\label{CebCond2}
\big|  q^k+1 - k\pi(k) \big| \le 2g' \frac{q^{k/2}}k .
\end{equation}
\begin{equation}\label{CebCond3}
2g'=-2\#G + \deg(D')+2.
\end{equation}
After some manipulations one obtains
$$  \pi_C(k)\ge \frac1{\#G} \frac{q^{k}+1}k - \frac{q^{k/2}}{k} \big( \frac{\deg(D)}2+2+\frac2{\#G}\big) - \big(1+\frac1k\big)\#G\frac{\deg(D)}2. $$
To ensure that the right hand side is positive for some (even) $k$, it thus suffices that 
\begin{equation}\label{CebCond4}
f(k):=q^{k} - q^{k/2} \Big( 2^{l-1}(\deg(r)+5) +2\Big) - \big(k+1 \big)2^{2l-1}(\deg(r)+1) >0. 
\end{equation}
We know that $l$ is the number of prime factors of $r$ and hence that $l\le\deg(r)$. There are at most $q$ places of degree $1$ and so for small $q$ such as $3,5,7$, already for small $l$ the degree of $r$ must be quite a bit larger than $l$. For instance if $l\ge7$ and $q=3$, then $\deg(r)\ge 3l-9$. Using these considerations and simple analysis on $f(k)$, it is simple if tedious to obtain the lower bounds in the table. We leave details to the reader.
\end{proof}

\begin{Prop}\label{ExplicitD}
For $\alpha$ as in Lemma~\ref{LemAlphaExists}, the quaternion algebra $D:= \QuatAlg{\alpha}{r}{K}$ is ramified exactly at $R$.
\end{Prop}
\begin{proof}
The proof is an immediate consequence of Proposition~\ref{PropComputationHilbertSymbol}.
\end{proof}

Since $r$ is a square modulo $\alpha$, there are 
$\epsilon, \nu \in A$ with $\deg(\epsilon) < \deg(\alpha)$ and $\epsilon^2 = r + \nu \alpha$.
\begin{Prop}\label{TheMaxOrderLambda}
$\Lambda:=A+Ai+Aj+A\frac{\epsilon i + i j}\alpha$ is a maximal $A$-order of $D$.
\end{Prop}
\begin{proof}
This follows easily from Proposition~\ref{DiscCritForMaxOrder} by computing the discriminant of the $A$-basis given.
\end{proof}

Since $\alpha$ has even degree and is monic, there exists a square root of $\alpha$ in $K_\infty$. We choose one and denote it by $\sqrt{\alpha}$. The following result provides an explicit realization for the embedding in Lemma~\ref{EmbeddingAbstract}(c).
\begin{Lem}\label{LemEmbedding}
The $K$-algebra homomorphism $\iota : D \to M_2(K_\infty)$ defined by $i \mapsto \begin{pmatrix} \sqrt{\alpha} & 0 \\ 0 & -\sqrt{\alpha} \end{pmatrix}$ and $j \mapsto \begin{pmatrix} 0 & 1 \\ r & 0 \end{pmatrix}$ induces an isomorphism $D \otimes_K K_\infty
 \cong M_2(K_\infty)$.
\end{Lem}
\begin{proof} 
One verifies $\iota(i)^2 = \alpha, \iota(j)^2 = r$ and $\iota(i) \iota(j) = - \iota(j) \iota(i)$ by an explicit calculation.
\end{proof}

\section{Facts about quaternion quotient graphs}
\label{Sec-QuatGraph}

In Section~\ref{Sec-BT-Tree} we have described the natural action of $\GL_2(K_\infty)$ on the Bruhat-Tits tree $\CT$. In the previous section, starting from $D$ as in Assumption~\ref{MainAss}, we have produced a discrete subgroup $\Gamma\subset\GL_2(K_\infty)$, the unit group of a maximal order. In this section we gather some known results about the induced action of $\Gamma$ on $\CT$ and the quotient graph $\Gamma \backslash \MCT$. We mainly follow~\cite{Papikian}.

\begin{Lem}[{{\cite[Cor.~to~Prop~II.1]{Serre1}}}]\label{LemDistanceEven}
For $v \in \Ver(\MCT)$ and $\gamma \in \Gamma$, the distance $d(v, \gamma v)$ is even.
\end{Lem}

\begin{Prop}[{{\cite[Lem.~5.1]{Papikian}}}]\label{PropGraphFinite}
The graph $\Gamma \backslash \MCT$ is finite graph.
\end{Prop}

\begin{Prop}[{{\cite[Prop.~5.2]{Papikian}}}]\label{PropStabilisors}
Let $v \in \Ver(\MCT)$ and $e \in \Edg(\MCT)$. Then $\Gamma_v := \Stab_{\Gamma}(v)$ is either isomorphic to $\BF_q^*$
or $\BF_{q^2}^*$. $\Gamma_e := \Stab_{\Gamma}(e)$ is isomorphic to $\BF_q^*$.
\end{Prop}

Note that the scalar matrices with diagonal in $\BF_q^*$ are precisely the scalar matrices in $\Gamma$. Clearly they act trivially on $\CT$. Hence a stabilizer of a simplex is isomorphic to $\BF_q^*$ precisely if it is the set of scalar matrices with diagonal in~$\BF_q^*$. 

We define $\bar\Gamma$ to be the image of $\Gamma$ in $\PGL_2(K_\infty)$ -- after what we have just seen we have $\bar\Gamma\cong \Gamma/\BF_q^*$. Then $\bar\Gamma_v:=\Stab_{\bar\Gamma}(v)$ is either trivial or isomorphic to $\BF_{q^2}^*/\BF_{q}^*\cong \BZ/(q+1)$ and $\bar\Gamma_e := \Stab_{\bar\Gamma}(e)$ is always trivial.

\begin{Def}[{{\cite[II.2.9]{Serre1}}}]
We call a simplex $t$ {\em projectively stable} if $\bar\Gamma_t$ is trivial and {\em projectively unstable} otherwise. 
\end{Def}

\begin{Cor}\label{CorStabilisors}
Let $v \in \Ver(\MCT)$ be projectively unstable. Then $\Gamma_v$ acts transitively on the vertices
adjacent to $v$.
\end{Cor}

Let $$\odd(R) := \begin{cases} 0 & \text{ if some place in $R$ has even degree,}\\ 1 & \text{ otherwise} \end{cases}$$

and let $$g(R) := 1 + \frac{1}{q^2 - 1} \left(\prod_{\Fp \in R} (q_\Fp - 1)\right) - \frac{q}{q + 1} 2^{\#R - 1} \odd(R)$$
where $q_\Fp = q^{\deg(\Fp)}$.
Let $\pi : \MCT \to \Gamma \backslash \MCT$ be the natural projection.

\begin{Thm}[{{\cite[Thm.~5.5]{Papikian}}}]
\label{ThmQuotStr}
\begin{enumerate}
\item The graph $\Gamma \backslash \MCT$ has no loops.
\item $h_1(\Gamma \backslash \MCT) = g(R)$.
\item For $\bar v \in \Gamma \backslash \MCT$ and $v \in \pi^{-1} (\bar v)$ we have: 
\begin{enumerate}
\item $v$ is projectively stable if and only if $\bar v$ has degree $q + 1$.
\item $v$ is projectively unstable if and only if $\bar v$ is terminal.
\end{enumerate}
\item Let $V_1$ (resp. $V_{q + 1}$) be the number of terminal (resp. degree $q + 1$) vertices of $\Gamma
\backslash \MCT$. Then
$$V_1 = 2^{\#R - 1} \odd(R) \text{ and } V_{q + 1} = \frac{1}{q - 1}(2 g(R) - 2 + V_1).$$
\end{enumerate}
\end{Thm}

\section{An algorithm to compute a fundamental domain}
\label{Sec-Algo}

Let the notation $\CT$, $\Gamma$ be as in the previous section.
\begin{Def}[{{\cite[\S~I.3]{Serre1}}}]\label{DefTreeOfReprs}
Let $G$ be a group acting on a graph~$\MCX$. A {\em  tree of representatives of $\MCX \pmod G$} is a subtree $\CS\subset\CG$ whose image in $G \backslash \MCX$ is a maximal subtree.
\end{Def}

The following definition is basically  \cite[\S~I.4.1, Lem.~4]{Serre1}, see also \cite[\S~I.5.4, Thm.~13]{Serre1}. Note that (a) differs from \cite[\S~I.4.1, Def.~7]{Serre1}.
\begin{Def}\label{DefFundamentalDomainWithSidePairing}
Let $G$ be a group acting on a tree~$\MCX$.
\begin{enumerate}
\item A {\em fundamental domain for $\CX$ under $G$} is a pair $(\CS,\CY)$ of subgraphs $\CS\subset \CY\subset\MCX$ such that 
\begin{enumerate}
\item $\CS$ is a tree of representatives of $\CX\pmod G$,
\item the projection $\Edg(\MCY) \to \Edg(G \backslash \MCX)$ is a bijection, and
\item any edge of $\CY$ has at least one of its vertices in~$\CS$. 
\end{enumerate}
\item An {\em edge pairing for a fundamental domain $\MCY$ of $\CX$ under $G$} is a map
\[ \EP:=\EP_{(\CS,\CY)}:=\{ e \in \Edg(\MCY)\setminus\Edg(\CS)\mid o(e) \in \MCS\} \to G: e\mapsto g_e\]
such that $g_e t(e) \in \Ver(\MCS)$. We write $\EP$ for {\em paired edges}. To avoid cumbersome notation, we usually abbreviate $\EP_{\CY,\CS}$ by $\EP$. 
\item An {\em enhanced fundamental domain for $\CX$ under $G$} consists of a fundamental domain, an edge pairing and simplex labels $G_t:=\Stab_G(t)$ for all simplices $t$ of $\CY$. 
\end{enumerate}
\end{Def}
An edge pairing encodes that under the $G$-action any $e=(v,v')\in\EP$ is identified (paired) with $ge=(g_e v,g_ev')$ when passing from $\CX$ to $G\backslash\CX$. Because $\CX$ is a tree and the image of $\CS$ in $G\backslash\CX$ is a maximal subtree, each edge in $\Edg(\MCY)\setminus\Edg(\CS)$ has exactly one of its vertices in $\Ver(\CS)$ and therefore $\EP$ contains exactly those edges of $\Edg(\MCY)\setminus\Edg(\CS)$ pointing away from~$\CS$. An enhanced fundamental domain is a graph of groups in the sense of \cite[I.4.4, Def.~8]{Serre1} realized inside~$\CX$. Given a fundamental domain with an edge pairing the tree $\CS$ can be recovered from $\CY$ and~$\EP$. 

\begin{Rem}
If one barycentrically subdivides $\MCT$, an alternative way to think of an edge pairing is that it pairs the two half sides $[o(e), \frac{1}{2} o(e) + \frac{1}{2} t(e)]$ and $[g_e t(e), g_e(\frac{1}{2} o(e) + \frac{1}{2} t(e))]$.
\end{Rem}

It will be convenient to introduce the following notation:
\begin{Def}
For any group $G$ acting on a set $X$ we define a category $\CC_G(X)$ whose objects are the elements of $X$ and whose morphism sets are defined as  
$$\Hom_G(x, y) := \lbrace \gamma \in G \mid g x = y \rbrace \subseteq G.$$
for $x,y\in X$. The composition of morphisms is given by multiplication in $G$.
\end{Def}
In particular $\End_G(x) := \Hom_G(x, x) = \Stab_G(x)$.

\smallskip

For the remainder of this section, we assume that $\Hom_\Gamma(v,w)$ can be computed effectively for
all $v, w \in \Ver(\MCT)$. This will be verified in Section~\ref{SecCompHom}. 

\begin{Alg} {\bf(Computation of the quotient graph)}\label{AlgQuotGraphNew}\\[.5em]
{\em Input:} A subgroup $\Gamma\subset\GL_2(K_\infty)$ for which there exists a routine for computing $\Hom_\Gamma(v,v')$ for all $v,v'\in\Ver(\CT)$ which are equidistant from $[L(0,0)]$. \\[.5em]
{\em Output:} A directed multigraph $\CG$ with a label attached to each simplex. The label values on edges are either $(e,1)$ (preset), or $(e,-1)$, or a pair $(e,g)$ with $e\in\Edg(\CT)$, $g\in\Gamma$. The label values on vertices are either $(v,1)$ (preset) or $(v,G)$ for $v\in\Ver(\CT)$ and $G\subset \Gamma$ a finite subgroup.\\[.5em]
{\em Algorithm:} 
\end{Alg}

\begin{enumerate}
\item Set $v_0=[L(0,0)]$. If $\#\End_\Gamma(v_0)=q^2-1$, replace $v_0$ by $[L(1,0)]$. If after replacement we still have $\#\End_\Gamma(v_0)=q^2-1$, then terminate the algorithm with the output the connected graph on $2$ vertices and one edge and with vertex labels $\End_\Gamma(v)$ for each of the two vertices~$v$.
\item Initialize a graph $\CG$ with $\Ver(\CG) = \lbrace v_0 \rbrace$ and $\Edg(\CG) = \emptyset$. Also, initialize lists $L:=(e\in\Edg(\CT)\mid o(e)=v_0)$, the edges adjacent to $v_0$, and $L':=\emptyset$. All vertices $v$ of $\CT$ are given by a matrix in vertex normal form $\vnf(v)$.
\item While $L$ is not empty:
\begin{enumerate}
\item For $i = 1$ to $\#L$ do:
\begin{enumerate}
 \item Let $e=(v, v^\prime)$ be the $i$th element in $L$.
\item Compute $\End_\Gamma(v^\prime)$.
\item If $\#\End_\Gamma(v^\prime) = q^2 - 1$ then:
\begin{enumerate} 
\item Add the vertex $v^\prime$ to $\Ver(\CG)$ and $e$ and $e^\star$ to $\Edg(\CG)$.
\item Store $(v',\End_\Gamma(v^\prime))$ as a vertex label for $v^\prime$.
\item Remove $e$ from $L$. 
\end{enumerate}
\item If $\#\End_\Gamma(v^\prime) = q - 1$, then for all $j<i$ do the following: 
\begin{enumerate}
\item Let $e'=(w, w^\prime)$ be the $j$th element in $L$.  
\item Compute $\Hom_\Gamma(v^\prime, w^\prime)$. 
\item If $\Hom_\Gamma(v^\prime, w^\prime) \neq \emptyset$, then do the following\\
$\bullet$ \ Add an edge $e'$ from $v$ to $w'$ to $\Edg(\CG)$, as well as its opposite.\\ 
$\bullet$ \ Give $e'$ the label $(e,g_e)$ for some $g_e\in\Hom_\Gamma(v^\prime, w^\prime)$ and give $e^{\prime\star}$ the label $(e,-1)$.\\
$\bullet$ \ Remove $(v, v^\prime)$ from $L$ and set $j:=i$.\\ 
$\bullet$ \ Remove $(w',\vnf(g_ev))$ from $L'$.\\
$\bullet$ \ If now $\degree_\CG(w')=q+1$, then remove $(w, w^\prime)$ from $L$.
\item Continue with the next $j$.
\end{enumerate}
\item If at the end of the $j$-loop we have $j=i$, then:
\begin{enumerate}
\item Add $v'$ to $\Ver(\CG)$, add $e$ and $e^\star$ to $\Edg(\CG)$. 
\item For all adjacent vertices $w \neq v$ of $v^\prime$ in $\MCT$ add $(v^\prime, w)$ to $L^\prime$.
\end{enumerate}
\end{enumerate}
\item Set $L := L^\prime$ and $L^\prime := \emptyset$.
\end{enumerate}
\item If $L$ is empty, return $\CG$.
\end{enumerate}

\begin{Rem}
One could randomly choose a vertex $[L(n,g)]$ as $v_0$ and replace it by $[L(n+1,g)]$, if it is projectively unstable. In this case, one would need to change the input of Algorithm~\ref{AlgQuotGraphNew} accordingly.
\end{Rem}
\begin{Rem}
The vertex label $(v,1)$ is used at all projectively stable vertices. For these, the stabilizer is the center of $\GL_2(K_\infty)$ intersected with $\Gamma$. There is no need to store this group each time. The same remark applies to all edge labels $(e,1)$. 

A maximal subtree $\CS$ of $\CG$ consists of all vertices and those edges with edge label $(e,1)$. It is completely realized within $\CT$. 

The edges with label $(e,g)$ are the edges which occur (ultimately) in $\EP$. The edge label $(e,-1)$ indicates that the opposite edge has a label $(e,g)$.
It is clear that the vertex and edge label allow one to easily construct an enhanced fundamental domain $(\CS,\CY)$ with an edge pairing and labels for the action of $\Gamma$ on~$\CT$.
\end{Rem}

\begin{Thm}\label{FirstAlgoSucceeds}
Suppose $\Gamma$ from Algorithm~\ref{AlgQuotGraphNew} satisfies the following conditions:
\begin{enumerate}
\item $d(v,g v)$ is even for all $g\in\Gamma,v\in\Ver(\CT)$,
\item for simplices $t$ of $\CT$ either $\bar\Gamma_t$ is trivial, or $t$ is a vertex and $\bar\Gamma_t\cong\BZ/(q+1)$,
\item $\Gamma\backslash\CT$ is finite.
\end{enumerate}
Then Algorithm~\ref{AlgQuotGraphNew} terminates and computes an enhanced fundamental domain for $\CT$ under~$\Gamma$.

By the results in Section~\ref{Sec-QuatGraph}, hypotheses (a)--(c) are satisfied if $\Gamma$ is the unit group of a maximal order of a quaternion algebra $D$ as in Assumption~\ref{MainAss}.
\end{Thm}

\begin{Rem}\label{RemOnGenOfAlgQuotGraph}
Let us comment on the hypotheses made in Theorem~\ref{FirstAlgoSucceeds} so that Algorithm~\ref{AlgQuotGraphNew} terminates: Following the example set by the number field case, it seems natural to consider the following situation: Let $C$ be a smooth projective geometrically connected curve over $\BF_q$, let $S$ be a finite set of closed points of $C$, set  $\Aa=\Gamma(C\setminus S,\CO_C)$ and let $K=Q(\Aa)$ be the fraction field of $\Aa$. Let furthermore $D$ denote a division algebra over $K$ which is ramified at all but one point $\infty$ of $S$ and let $\Lambda$ be an $\Aa$-order. As can be deduced in this situation from \cite{Vigneras} by an argument similar to Proposition~\ref{PropImageDiscrete}, the group $\Lambda^*$ modulo its center acts discretely on the Bruhat-Tits tree for $\PGL_2(K_\infty)$. We expect but have not checked that hypotheses~(1)--(3) of Theorem~\ref{FirstAlgoSucceeds} are always met in this situation. What is missing in this general situation is an explicit algorithm to compute $\Hom_{\Lambda^*}(v,v')$. For this, see Remark~\ref{RemOnHomGvw}.
\end{Rem}

\begin{proof}[Proof of Theorem~\ref{FirstAlgoSucceeds}]
Let $\MCG$ be the output of Algorithm~\ref{AlgQuotGraphNew}.
We show that any two distinct simplices of $\MCG$ have labels $(t,?)$ and $(t',?)$ with $t'\notin\Gamma t$ and that for all
simplices $t$ of $\MCT$ there is a simplex of $\MCG$ whose label is $(t',?)$ for some $t'\in \Gamma t$.

For the first assertion, let $v_1, v_2 \in \Ver(\CT)$ be distinct first entries in labels of vertices of $\CG$ and suppose that $\gamma v_1 = v_2$ for some $\gamma \in \Gamma\setminus\Gamma_{v_1}$. We seek a contradiction. In a first reduction step we show that we may assume that $v_1$ is projectively stable: So suppose $v_1$ is projectively unstable. Then since 
\begin{eqnarray}\label{eq_stabilisors}
\Stab_\Gamma(v_2) = \gamma \Stab_\Gamma(v_1) \gamma^{-1},
 \end{eqnarray}
also $v_2$ has to be projectively unstable. Hence both $v_1$ and $v_2$ are terminal vertices in $\CG$. Let $v_1^\prime$ and $v_2^\prime$ be their unique adjacent vertices in $\CG$. Since $v_1^\prime$ is adjacent to $v_1$, it follows that $\gamma v_1^\prime$ is adjacent to $\gamma v_1 = v_2$. By condition (b) the stabilizer $\Stab_\Gamma(v_2)$ acts transitively on the vertices adjacent to $v_2$. Hence there exists $\gamma^\prime\in \Stab_\Gamma(v_2)$ with 
\begin{equation}\label{Homv1'v2'}
\gamma^\prime \gamma v_1^\prime = v_2^\prime,
\end{equation} 
and so $v_1^\prime$ and $v_2^\prime$ are $\Gamma$-equivalent. If $v_1^\prime$ and $v_2^\prime$ were also projectively unstable and therefore terminal vertices in $\CG$, then, since $\CG$ is connected, $\CG$ would have to be the graph consisting of the two vertices $v_1, v_2$ and one edge connecting them. This contradicts condition~(a). Therefore $v_1^\prime$ and $v_2^\prime$ must be projectively stable and $\Gamma$-equivalent. To conclude the reduction, observe that we cannot have $v_1'=v_2'$, since in this case we must have $\gamma'\gamma\in\BF_q^*$ from (\ref{Homv1'v2'}). But $\gamma'\gamma $ maps $v_1$ to $v_2$ and this would contradict $v_1\neq v_2$.

Now suppose $v_1$ is projectively stable. Then by equation~\eqref{eq_stabilisors} so is $v_2$. Let $v$ be the initial vertex of the algorithm and let $i_1 = d(v, v_1)$ and $i_2 = d(v, v_2)$. We prove the assertion by induction over $i_1$: If $i_1 = 1$ then also $i_2 = 1$  because of condition (a). Hence the vertices $v_1$ and $v_2$ both have the same distance $1$ from $v$ and since $\Hom_\Gamma(v_1, v_2) = q - 1$, Algorithm~\ref{AlgQuotGraphNew} with the first choice of $L$ rules out that they both lie in $\CG$. This is a contradiction. The same reasoning rules out $i_1 = i_2$ for any $i_1, i_2 \geq 1$.

Suppose $i_1 > 1$. By condition~(a) and the previous line we may assume $i_1 = i_2 + 2 m$ for some $m \in \BZ_{\geq 1}$. Let $v_1^\prime$ be the vertex on the geodesic from $v_1$ to $v$ so that $d(v, v_1^\prime) = i_1 - 1$. Then by the construction of $\CG$ we have $v_1^\prime \in \CG$.  The vertex $\gamma v_1^\prime$ is adjacent to $\gamma v_1 = v_2$. Now observe that $\gamma v_1^\prime$ does not belong to $\CG$ because otherwise we could apply the  induction hypothesis to $v_1^\prime, \gamma v_1^\prime$, using $d(v_1^\prime, v) = i_1 - 1$ and $d(\gamma v_1^\prime, v) \leq i_2 + 1$ to obtain a contradiction.

It follows that $v_2^\prime := \gamma v_1^\prime \not \in \CG$. Since by construction the geodesic from $v$ to $v_2$ lies on $\CG$ we have $d(v_2^\prime, v) = i_2 + 1$. Now by  the algorithm that defines $\CG$ the vertex $v_2'$ must be equivalent to a vertex of distance $i_2-1$, i.e., there are $\gamma^\prime \in \Gamma, v_2^{\prime \prime} \in \CG$ with $d(v_2^{\prime \prime}, v) = i_2 - 1$ such that $v_2^{\prime \prime} = \gamma^\prime v_2^\prime$. But then we apply the induction hypothesis to $v_1^\prime, v_2^{\prime \prime}$ and again obtain a contradiction. This concludes the proof of the first assertion for vertices.

Suppose now that $e=(v_0,v_1), e^\prime(v'_0,v'_1)$ occur as first entries in $\Edg(\CG)$, lie in the same $\Gamma$-orbit, are distinct and occur within edge labels of $\CG$. Let $\gamma$ be in $\Gamma$ with $e^\prime = \gamma e$. Note that not all the vertices $v_i$ and $v_i'$ must occur in vertex labels from $\CG$ but each edge must at least have one vertex that does -- see step (c)(i)4.C. Suppose after possibly changing the orientation of edges and the indices that $v_0$ has minimal distance from $v$. By construction of $\CG$ the vertex $v_0$ occurs in a vertex label. If $v_0'=\gamma v_0$ occurs in a vertex label of $\CG$, then by the case already treated, we must have $v_0=v_0'$. Since $e\neq e'$ it follows that $v_0$ is projectively unstable. But then the algorithm does not yield an edge starting at $v_0$ and ending at a vertex $v_1$ with $d(v,v_1)>d(v,v_0)$. This is a contradiction. 

It follows that $v_0'=\gamma v_0$ does not occur in a vertex label. Hence $v_1'$ must occur in a vertex label. By essentially the argument just given, $v_1$ can also not occur in a vertex label. Hence $(e,\gamma)$ must be an edge label and moreover $d(v,v_1')=d(v,v_0)+1=d(v,v_0')-1$. But then in step (c)(i)4.C of Algorithm~\ref{AlgQuotGraphNew} the edge $e'$ must have been removed from the list $L'$ and so it cannot occur in a label of an edge of $\CG$.

\smallskip

We finally come to the second assertion: By construction, $\CG$ defines a connected subgraph of $\Gamma \backslash \MCT$, since we already showed that there are no $\Gamma$-equivalent simplices in $\CG$. Moreover, at any vertex of this subgraph the degree within $\CG$ and within $\Gamma \backslash \MCT$ is the same. Hence $\CG$ defines a connected component of $\Gamma \backslash \MCT$. But
$\MCT$ and hence $\Gamma \backslash \MCT$ are connected and thus $\CG = \Gamma \backslash \MCT$. 
\end{proof}

We further describe an algorithm to compute for any $v^\prime \in \Ver(\MCT)$ a
$\Gamma$-equivalent vertex $v^{\prime \prime} \in \MCG$. This can be done
in time linear to the distance from $v^\prime$ to $\MCG$. For this algorithm
we need the stabilizers of the terminal vertices of $\MCG$ and the elements $\gamma \in \Hom_\Gamma(v_i, v_j)$,
which we both stored as vertex and edge labels during the computation of $\MCG$. We call this
algorithm the reduction algorithm.
We need to be able to do the following:
\begin{enumerate}
\item Find the geodesic from $v^\prime$ to $v$. This was discussed in Remark~\ref{RemGeodesivToStandVertex}.
\item Determine the extremities of a given geodesic in $\MCG$. Since the vertices in $\MCG$ are all stored in 
the vertex normal form, this can be done in constant time.
\end{enumerate}

\begin{Alg}\label{AlgReduction} {\bf (The reduction algorithm)}\\[.2em]
{\em Input:} $v^\prime\in\Ver(\MCT)$ and $\MCG$ the output of Algorithm~\ref{AlgQuotGraphNew} with initial vertex $v$.\\[.2em]
{\em Output:} A tuple $(w, \gamma) \in \Ver(\MCG) \times \Gamma$ with $v^\prime = \gamma w$.\\[.2em]
{\em Algorithm:}
\end{Alg}

\begin{enumerate}
 \item Let $\MCT_0 : (v^\prime = v_m, v_{m - 1}, \dots, v)$ be the geodesic  
from $v^\prime$ to $v$. Let $v_i$ be the vertex of $\MCT_0 \cap \MCG$ closest to $v'$. 
Let $r = m - i$, this is the distance from $v^\prime$ to $\MCG$.

\item If $r = 0$, we have $v^\prime \in \MCG$. Then return $(v^\prime, 1)$.

\item If $r > 0$, we distinguish two cases:

\begin{enumerate}
 \item If $v_i$ is projectively unstable, by a for-loop through the elements $\gamma$ in $\Stab_\Gamma(v_i)$, find an element $\gamma\in\Gamma$ such that $\gamma v_{i + 1}$ is a vertex of $\MCG$. Replace $v^\prime$ by $\gamma v^\prime$ and apply the algorithm recursively to get some
pair $(w, \tilde \gamma)$ in $\Ver(\MCG) \times \Gamma$. Return $(w, \tilde \gamma \gamma )$.

\item If $v_i$ is projectively stable, run a for-loop through the vertices $\tilde v$ in $\MCG$ adjacent to
$v_i$ to find the unique $\tilde v$ such that either: (i), the edge label of the edge from $\tilde v$ to $v_i$ is of the form $(e,\gamma)$ for some $\gamma\in\Gamma$ with $\gamma t(e)=v_i$ and $\gamma o(e)=v_{i+1}$, or (ii), the edge label from $v_i$ to $\tilde v$ is of the form $(e,\gamma)$ for some $\gamma\in\Gamma$ with $o(e)=v_i$ and $t(e)=v_{i+1}$. In case (i),  replace $v'$ by $\gamma^{-1} v'$ and apply the algorithm recursively to get some pair $(w,\tilde\gamma)$ for $\gamma^{-1} v'$. Return $(w,\tilde\gamma\gamma^{-1})$. In case (ii), replace $v'$ by $\gamma v'$ and apply the algorithm recursively to get some pair $(w,\tilde\gamma)$ for $\gamma v'$. Return $(w,\tilde\gamma\gamma)$.
\end{enumerate}
\end{enumerate}

\begin{Prop} Let $v^\prime$ in $\MCT$ and let $\MCG$ be the output of Algorithm~\ref{AlgQuotGraphNew} under the hypothesis of Theorem~\ref{FirstAlgoSucceeds} with initial vertex $v$.
Then Algorithm~\ref{AlgReduction} computes a $\Gamma$-equivalent vertex $w$ of $v^\prime$ and an element
$\gamma \in \Gamma$ with $\gamma v^\prime = w$. It requires $\MCO(n^3 \deg(r)^2)$ additions and multiplications
in $\BF_q$ where $n$ is the distance of $v^\prime$ to $\MCG$.
\end{Prop}

\begin{proof}
In both cases of the algorithm we find an edge label that moves $v^\prime$ closer to $\MCG$. 
Since each step of the algorithm decreases the distance $d(v^\prime, \MCG)$, the algorithm terminates
after at most $n$ steps. From Corollary~\ref{CorSize} and Proposition~\ref{PropChebtor} it follows that at step $j$ one multiplies a matrix of height $(j-1)\frac{5}{2} \deg(r)$ with one of height $\frac{5}{2} \deg(r)$. Further one has to compute the vertex normal form of a matrix of height at most $(j+1) \frac{5}{2} \deg(r)$. This takes at most $(8j+8j^2) \frac{5}{2}^2 \deg(r)^2$ operations in $\BF_q$. Summing over $j$, the asserted bound follows.
\end{proof}

\begin{Ex}\label{ExQuotGraph}
In Figure~\ref{fig:AlgQuotGraphExample} we give an example of the Algorithm~\ref{AlgQuotGraphNew}, where $q=5$ and
$ r=T(T+1)(T+2)(T+3)$. We start with $\begin{pmatrix} 1/T & 0 \\ 0 & 1 \end{pmatrix}$ as the initial vertex $v$.
The adjacent vertices correspond to the matrix $\begin{pmatrix} 1 & 0 \\ 0 & 1 \end{pmatrix}$, which is 
a terminal vertex, and the five matrices $\begin{pmatrix} 1/T^2 & \alpha 1/T \\ 0 & 1 \end{pmatrix}$ 
with $\alpha \in \BF_5$. Using the algorithm described in Section~\ref{SecCompHom} we compute that $\begin{pmatrix} 1/T^2 & 0 \\ 0 & 1 \end{pmatrix}$ is the only
projectively unstable vertex and 
$$\#\Hom_\Gamma( \begin{pmatrix} 1/T^2 & 1/T \\ 0 & 1 \end{pmatrix}, 
\begin{pmatrix} 1/T^2 & 4 1/T \\ 0 & 1 \end{pmatrix}) = 4,$$
 
$$\#\Hom_\Gamma( \begin{pmatrix} 1/T^2 & 2 \pi \\ 0 & 1 \end{pmatrix}, 
\begin{pmatrix} 1/T^2 & 3 \pi \\ 0 & 1 \end{pmatrix}) = 4.$$
This finishes Step 1 of the algorithm, as depicted in Figure~\ref{fig:AlgQuotGraphExample}. In Step 2 we
 then continue with the eight indicated vertices of level $3$. In this case, the algorithm terminates after 3 steps.
\end{Ex}

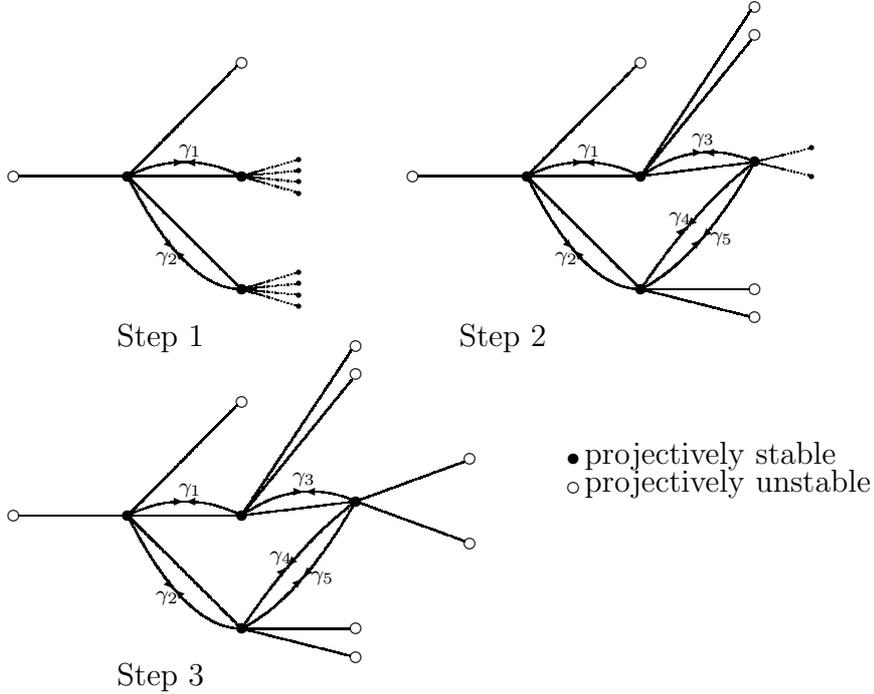
\begin{figure}[ht]
\begin{picture}(15, 12)
 
\put(3, 6){Step 1}
\put(9, 6){Step 2}
\put(3, 0){Step 3}

\put(11,4){\circle*{0.2}}
\put(11.2,3.92){projectively stable}

\put(11,3.5){\circle{0.2}}
\put(11.2,3.42){projectively unstable}

\put(0.5,6.5){\setlength{\unitlength}{0.75cm}
\begin{picture}(6,5)
\put(2.5,2.5){\circle*{0.2}}
\put(0.5,2.5){\circle{0.2}}
\put(4.5,2.5){\circle*{0.2}}
\put(4.5,4.5){\circle{0.2}}
\put(4.5,0.5){\circle*{0.2}}
\qbezier(0.6,2.5)(0.6,2.5)(2.5,2.5)
\qbezier(2.5,2.5)(2.5,2.5)(4.5,2.5)
\qbezier(2.5,2.5)(2.5,2.5)(4.5,0.5)
\qbezier(2.5,2.5)(2.5,2.5)(4.42,4.45)
\qbezier(2.5,2.5)(3.5,3)(4.5,2.5)

\put(3.5,2.74){\vector(1,0){0}}
\put(3.5,2.74){\vector(-1,0){0}}

\put(3.32,1.2){\vector(1,-1){0}}
\put(3.32,1.2){\vector(-1,1){0}}
\put(3.4,2.9){\begin{scriptsize}$\gamma_1$\end{scriptsize}}
\put(3,1){\begin{scriptsize}$\gamma_2$\end{scriptsize}}

\qbezier(2.5,2.5)(3.5,0.5)(4.5,0.5)

\put(5.5,2.8){\circle*{0.1}}
\put(5.5,2.6){\circle*{0.1}}
\put(5.5,2.4){\circle*{0.1}}
\put(5.5,2.2){\circle*{0.1}}

\qbezier[30](4.5,2.5)(4.5,2.5)(5.5,2.8)
\qbezier[30](4.5,2.5)(4.5,2.5)(5.5,2.6)
\qbezier[30](4.5,2.5)(4.5,2.5)(5.5,2.4)
\qbezier[30](4.5,2.5)(4.5,2.5)(5.5,2.2)

\put(5.5,0.8){\circle*{0.1}}
\put(5.5,0.6){\circle*{0.1}}
\put(5.5,0.4){\circle*{0.1}}
\put(5.5,0.2){\circle*{0.1}}

\qbezier[30](4.5,0.5)(4.5,0.5)(5.5,0.8)
\qbezier[30](4.5,0.5)(4.5,0.5)(5.5,0.6)
\qbezier[30](4.5,0.5)(4.5,0.5)(5.5,0.4)
\qbezier[30](4.5,0.5)(4.5,0.5)(5.5,0.2)
\end{picture} }

\put(7.5,6.5){\setlength{\unitlength}{0.75cm}
\begin{picture}(6,5)
\put(2.5,2.5){\circle*{0.2}}
\put(0.5,2.5){\circle{0.2}}
\put(4.5,2.5){\circle*{0.2}}
\put(4.5,4.5){\circle{0.2}}
\put(4.5,0.5){\circle*{0.2}}
\qbezier(0.6,2.5)(0.6,2.5)(2.5,2.5)
\qbezier(2.5,2.5)(2.5,2.5)(4.5,2.5)
\qbezier(2.5,2.5)(2.5,2.5)(4.5,0.5)
\qbezier(2.5,2.5)(2.5,2.5)(4.42,4.45)
\qbezier(2.5,2.5)(3.5,3)(4.5,2.5)

\put(3.5,2.74){\vector(1,0){0}}
\put(3.5,2.74){\vector(-1,0){0}}

\put(3.32,1.2){\vector(1,-1){0}}
\put(3.32,1.2){\vector(-1,1){0}}
\put(3.4,2.9){\begin{scriptsize}$\gamma_1$\end{scriptsize}}
\put(3,1){\begin{scriptsize}$\gamma_2$\end{scriptsize}}

\qbezier(2.5,2.5)(3.5,0.5)(4.5,0.5)

\put(6.5,5.5){\circle{0.2}}
\put(6.5,5){\circle{0.2}}
\put(6.5,0.5){\circle{0.2}}
\put(6.5,0){\circle{0.2}}

\qbezier(4.5,2.5)(4.5,2.5)(6.43,5.43)
\qbezier(4.5,2.5)(4.5,2.5)(6.43,4.9)
\qbezier(4.5,0.5)(4.5,0.5)(6.4,0.5)
\qbezier(4.5,0.5)(4.5,0.5)(6.4,0.02)

\put(6.5,2.75){\circle*{0.2}}

\qbezier(4.5,2.5)(4.5,2.5)(6.5,2.75)
\qbezier(4.5,2.5)(5.5,3.2)(6.5,2.75)
\qbezier(4.5,0.5)(5.5,2)(6.5,2.75)
\qbezier(4.5,0.5)(5.5,1)(6.5,2.75)

\put(7.5,3){\circle*{0.1}}
\put(7.5,2.5){\circle*{0.1}}
\qbezier[30](6.5,2.75)(6.5,2.75)(7.5,3)
\qbezier[30](6.5,2.75)(6.5,2.75)(7.5,2.5)

\put(5.6,2.92){\vector(1,0){0}}
\put(5.6,2.92){\vector(-1,0){0}}

\put(5.3,1.6){\vector(-1,-1){0}}
\put(5.3,1.6){\vector(1,1){0}}
\put(5.57,1.4){\vector(-1,-1){0}}
\put(5.57,1.4){\vector(1,1){0}}

\put(5.4,3.1){\begin{scriptsize}$\gamma_3$\end{scriptsize}}
\put(5,1.7){\begin{scriptsize}$\gamma_4$\end{scriptsize}}
\put(5.72,1.35){\begin{scriptsize}$\gamma_5$\end{scriptsize}}

\end{picture} }

\put(0.5,0.5){\setlength{\unitlength}{0.75cm}
\begin{picture}(6,5)

\put(2.5,2.5){\circle*{0.2}}
\put(0.5,2.5){\circle{0.2}}
\put(4.5,2.5){\circle*{0.2}}
\put(4.5,4.5){\circle{0.2}}
\put(4.5,0.5){\circle*{0.2}}
\qbezier(0.6,2.5)(0.6,2.5)(2.5,2.5)
\qbezier(2.5,2.5)(2.5,2.5)(4.5,2.5)
\qbezier(2.5,2.5)(2.5,2.5)(4.5,0.5)
\qbezier(2.5,2.5)(2.5,2.5)(4.42,4.45)
\qbezier(2.5,2.5)(3.5,3)(4.5,2.5)

\put(3.5,2.74){\vector(1,0){0}}
\put(3.5,2.74){\vector(-1,0){0}}

\put(3.32,1.2){\vector(1,-1){0}}
\put(3.32,1.2){\vector(-1,1){0}}
\put(3.4,2.9){\begin{scriptsize}$\gamma_1$\end{scriptsize}}
\put(3,1){\begin{scriptsize}$\gamma_2$\end{scriptsize}}

\qbezier(2.5,2.5)(3.5,0.5)(4.5,0.5)

\put(6.5,5.5){\circle{0.2}}
\put(6.5,5){\circle{0.2}}
\put(6.5,0.5){\circle{0.2}}
\put(6.5,0){\circle{0.2}}

\qbezier(4.5,2.5)(4.5,2.5)(6.43,5.43)
\qbezier(4.5,2.5)(4.5,2.5)(6.43,4.9)
\qbezier(4.5,0.5)(4.5,0.5)(6.4,0.5)
\qbezier(4.5,0.5)(4.5,0.5)(6.4,0.02)

\put(6.5,2.75){\circle*{0.2}}

\qbezier(4.5,2.5)(4.5,2.5)(6.5,2.75)
\qbezier(4.5,2.5)(5.5,3.2)(6.5,2.75)
\qbezier(4.5,0.5)(5.5,2)(6.5,2.75)
\qbezier(4.5,0.5)(5.5,1)(6.5,2.75)

\put(8.5,3.5){\circle{0.2}}
\put(8.5,2){\circle{0.2}}
\qbezier(6.5,2.75)(6.5,2.75)(8.4,3.45)
\qbezier(6.5,2.75)(6.5,2.75)(8.4,2.05)

\put(5.6,2.92){\vector(1,0){0}}
\put(5.6,2.92){\vector(-1,0){0}}

\put(5.3,1.6){\vector(-1,-1){0}}
\put(5.3,1.6){\vector(1,1){0}}
\put(5.57,1.4){\vector(-1,-1){0}}
\put(5.57,1.4){\vector(1,1){0}}

\put(5.4,3.1){\begin{scriptsize}$\gamma_3$\end{scriptsize}}
\put(5,1.7){\begin{scriptsize}$\gamma_4$\end{scriptsize}}
\put(5.72,1.35){\begin{scriptsize}$\gamma_5$\end{scriptsize}}

\end{picture} }

\end{picture}
\caption{Example: $q=5$,  $r = T(T+1)(T+2)(T+3)$}
\label{fig:AlgQuotGraphExample}
\end{figure}

\begin{Ex}
Consider $K = \BF_5(T)$ and the two discriminants $r_1 = (T^2 + T + 1)\cdot T \cdot (T+1) \cdot (T+2)$ and $r_2 = (T^2 + 2)\cdot T \cdot (T+1) \cdot (T+2)$. Let $\Gamma_i$ be the group of units of a maximal order of a quaternion algebra of discriminant $r_i$ for $i \in \{ 1, 2\}$. Then $\Gamma_1 \backslash \MCT$ has $14$ cycles of length $2$, while $\Gamma_2 \backslash \MCT$ has $10$ cycles of length $2$. Hence these two graphs are not isomorphic. This answers a question of Papikian who asked for an example in which the lists of degrees of the factors of $r$ and $r'$ are the same but where the graphs are non-isomorphic. This is similar to~\cite[Rem~2.22]{GekelerNonnengard} where congruence subgroups $\Gamma_0(\Fn)$ and $\Gamma_0(\Fn')$ of $\GL_2(\Aa)$ are considered.
\end{Ex}

\section{Computing $\Hom_\Gamma(v, w)$}\label{SecCompHom}

Let $r,\alpha,\Lambda\subset D=\QuatAlg{\alpha}{r}{K},\iota$ be as at the end of Section~\ref{Sec-QuatAlg}; recall also that $\pi=1/T$ is a uniformizer of $K_\infty$.

\begin{Lem}
To compute $\sqrt{\alpha}$ in $K_\infty = \BFq((\pi))$ to $n$ digits of accuracy one requires $\CO(n^3)$ additions and multiplications in $\BF_q$.
\end{Lem}
\begin{proof}
Let $m=\deg(\alpha)$. It suffices to compute the square root $u$ of the $1$-unit $\pi^m\alpha$ to $n$ digits accuracy. This can be done by the Newton iteration in $n$ steps starting with the approximation $u_0=1$. The $k$-th approximation is $u_{k}=u_{k-1}-\frac{u_{k-1}^2-\pi^m\alpha}{u_{k-1}}$. From the right hand expression one only needs to compute $u_{k-1}^2-\pi^m\alpha$ which requires $n^2$ operations in $\BF_q$. The $k$-th digit past the decimal point divided by $2$ has then to be subtracted from~$u_{k-1}$.
\end{proof}

To state the following result, we define a (logarithmic) height $\|\,\|$ on elements of $\Lambda$. Its definition will be in terms of our standard $A$-basis of $\Lambda$, and it will depend on this choice. For $(\lambda_1, \lambda_2, \lambda_3, \lambda_4)\in A^4$ we define
\begin{equation}\label{DefOfHeight}
\|  \lambda_1\cdot 1 +  \lambda_2\cdot i +  \lambda_3\cdot j +  \lambda_4\cdot  \frac{\eps i+ i j}{\alpha} \|:= \max_{ i \in\{1,2,3,4\}} \ \deg(\lambda_i). 
\end{equation}
We also define, as an abbreviation, $v_\infty$ applied to a matrix or a vector of elements in $K_\infty$ to be the minimum of all the $v_\infty$-valuations of all entries.
  
\begin{Thm}\label{PropHomCompCharOdd}
Suppose $v, v^{\prime}\in\Ver(\CT)$ have distance $n$ from $v_0 = [L(0, 0)]$.

\begin{enumerate}
\item There is an algorithm that computes $\Hom_\Gamma(v, v')$ in time $\MCO(n^4)$ field operations over $\BFq$.
\item All $\gamma \in \Hom_\Gamma(v, v')$ satisfy $\| \gamma \| \leq n + \deg(\alpha)/2$.
\end{enumerate}
\end{Thm}

\begin{proof}
If $v=[L(l,g)]$ has distance $n$ from $v_0$, then either
\begin{eqnarray*}
l = n &\text{ and }& \deg_l(g) \text{ lies in }\{0, \dots, n\} \text{ or}\\
l\in\{-n,-n+2,-n+4,\ldots,n-2\} &\text{ and }& \deg_l(g) = \frac{n + l}{2},
\end{eqnarray*}
see Figure~\ref{fig:BTTreeVertMatr} and Remark~\ref{RemGeodesivToStandVertex}. Moreover the path from $[L(0, 0)]$ to $[L(l, g)]$ is via $L(\frac{l - n}{2}, 0)$ if $l < n$ and via $L(n - \deg_l(g), 0)$ if $l = n$.  Set $n_1 := \deg_l(g)$ and $n_2 := n - n_1$ if $l=n$ and $n_2=n_1-n$ if $l<n$. In Figure~\ref{fig:BTTreeVertMatr}, the integers $n_1$ and $n_2\in\BZ$ are the coordinates of $v$ from the baseline toward it and along the baseline, respectively. Moreover $l=n_1+n_2$ and $g\in\pi^{l-n_1}\CO_\infty=\pi^{n_2}\CO_\infty$. Similarly we define $n_1'$ and $n_2'$ for $v'=[L(l',g')]$ which is also of distance $n$ from $v_0=[L(0,0)]$.

Let now $\gamma = \begin{pmatrix} \pi^{l} & g \\ 0 & 1 \end{pmatrix}$ and 
$\gamma^\prime = \begin{pmatrix} \pi^{l'} & g^\prime \\ 0 & 1 \end{pmatrix}$
be the matrices in vertex normal form representing $v$ and $v^{\prime}$ respectively. By definition of $\Hom_\Gamma$ we have
$$\Hom_{\Gamma}(v, v') = \gamma' \GL_2(\CO_\infty) K_\infty^* \gamma^{-1} \cap \Gamma.$$
Because $v_\infty(\det(\gamma)) = l, v_\infty(\det((\gamma^\prime)^{-1}) = l^\prime $, 
$v_\infty(\det(\sigma)) = 0$ for all $\sigma \in \GL_2(\CO_\infty)$ and $v(\det\Gamma) = \lbrace 0 \rbrace$, we see that
$$\Hom_{\Gamma}(v, v^\prime) = \pi^{(l - l^\prime)/2} \gamma' \GL_2(\CO_\infty) \gamma^{-1} \cap \Gamma,$$
where we simply write $\pi^{(l - l^\prime)/2}$ for the scalar matrix $\pi^{(l - l^\prime)/2} \cdot 1_2$. By taking determinants on both sides and using the fact that $\CO_\infty \cap{\Aa} = \BFq$, we finally obtain
\begin{equation}\label{HomEqn}
\Hom_{\Gamma}(v, v^\prime) \ \stackrel{{\scriptscriptstyle \bullet}}{\cup} \ \lbrace 0 \rbrace \ = \ \pi^{(l - l^\prime)/2} \gamma' M_2(\CO_\infty) \gamma^{-1} \ \cap \  \Lambda .
\end{equation}

Equation (\ref{HomEqn}) can be interpreted in the following way: The intersection in the previous line is up to change by conjugation the same as $M_2(\CO_\infty)  \cap\pi^{(l' - l)/2}  \gamma^{\prime\,-1} \Lambda  \gamma$. Here $M_2(\CO_\infty)$ is the unit ball in $M_2(K_\infty)$, a $K_\infty$-vector space of dimension $4$ and $\pi^{(l' - l)/2} \gamma^{\prime\,-1} \Lambda  \gamma$ is a discrete $A$-lattice (of rank $4$) in this vector space. I.e., we need to compute the shortest non-zero vectors of the lattice $\pi^{(l' - l)/2} \gamma^{\prime\,-1} \Lambda  \gamma$ with respect to the norm given by $M_2(\CO_\infty)$. If these vectors have norm at most one, they form $\Hom_{\Gamma}(v, v^\prime) $. If their norm is larger than one, then $\Hom_{\Gamma}(v, v^\prime) $ is empty. In particular, the problem can in principle be solved by the function field version of the LLL algorithm.

However, the implemented versions of the LLL algorithm \cite{Hess,Paulus} need an a priori knowledge of the precision by which $\alpha$ has to be computed as an element in $\BFq((\pi))$.  This in turn makes it necessary to find a bound on the height of the elements in $\Hom_{\Gamma}(v, v^\prime) $, if described as a linear combination in terms of our standard $A$-basis for $\Lambda$. 
Moreover, \cite{Hess,Paulus} do not give a complexity analysis for their algorithms.
To derive these quantities, i.e.\ precision, height and complexity, we proceed as follows. Set
$$C = \begin{pmatrix} 1 & \sqrt{\alpha} & 0 & \frac{\epsilon}{\sqrt{\alpha}} \\
			   0 & 0 & 1 & \frac{1}{\sqrt{\alpha}} \\
			   0 & 0 & r & \frac{-r}{ \sqrt{\alpha}} \\
			   1 & -\sqrt{\alpha} & 0 & \frac{-\epsilon}{\sqrt{\alpha}} 
           \end{pmatrix} \!\textrm{\ and\ } 
           B = \begin{pmatrix} \!\!\pi^{\frac{l'-l}2} &\!\! 0 &\!\! g' \pi^{\frac{-l'-l}2} &\!\! 0 \\
 \!\! -g \pi^{\frac{l'-l}2} &\!\! \pi^{\frac{l'+l}2} &\!\!  -g g' \pi^{\frac{-l'-l}2} &\!\! g' \pi^{\frac{l-l'}2} \\
			 \!\! 0 &\!\! 0 &\!\! \pi^{\frac{-l'-l}2} &\!\! 0 \\
			 \!\! 0 &\!\! 0 &\!\! -g \pi^{\frac{-l'-l}2} &\!\! \pi^{\frac{l-l'}2} 
          \end{pmatrix}.$$
Observe that $v_\infty(g\pi^{-\frac{l}2})\ge n_2-\frac{n_1+n_2}2=\frac{n_2-n_1}2 \ge -\frac{n}2$ and that $-|l|\ge -n$. This implies that $v_\infty(B)\ge -n$. Similarly, using $\deg(\eps)\le\deg(\alpha)$ and computing $C^{-1}$ explicitly, one finds $v_\infty(C^{-1})\ge -m$ where we abbreviate $m:=\frac{\deg(\alpha)}2\in\BZ_{\ge1}$.

We now flatten $2\times2$-matrices to column vectors of length $4$. Taking the explicit form of the $A$-basis of $\Lambda$ from Lemma~\ref{LemEmbedding} into account, as well as the explicit forms of $\gamma$ and $\gamma'$, the solutions to (\ref{HomEqn}) are the solution of the linear system of equations
\begin{equation} \label{HomEqn2}
C\Ulambda=B\Ux,
\end{equation} 
where $\Ulambda$ denotes a (column) vector in $A^4$ and $\Ux$ a (column) vector in $\CO_\infty^4$. The equivalent form $\Ulambda=C^{-1}B\Ux$ and the above estimates on the valuations of $C^{-1}$ and $B$ now immediately imply $v_\infty(\Ulambda)\ge-(n+m)$. In other words, the components of $\Ulambda$ are polynomials and $\|\Ulambda\|\le n+m$. This proves (b).

Next, consider (\ref{HomEqn2}) in the form $B^{-1}C\Ulambda=\Ux$. Again by explicit computation, we have $v_\infty(B^{-1})\ge-n$ and $v_\infty(C)\ge -\max\{\deg(r),m\}=:-d$. Writing $B^{-1}C=\sum_{k=-d}^\infty X_k\pi^k$ as a power series with $X_k\in M_4(\BFq)$ and using the bound from (b), equation (\ref{HomEqn2}) is equivalent to 
$$ {\textstyle\big(\sum_{k=-(n+m)}^{n+d} X_k\pi^{-k} \big)\Ulambda\equiv 0 \pmod {\CO_\infty^4}}.$$
We also expand $\Ulambda=\sum_{k=0}^{n+m}\Ulambda_k\pi^{-k}$ as a polynomial in $\pi^{-1}$ with $\Ulambda_k\in \BFq^4$ and let $X_k$ and $\Ulambda_k$ be zero outside the range of indices $k$ indicated above. Then (\ref{HomEqn2}) becomes equivalent to the system of linear equations
$$ {\textstyle\big(\sum_{k=0}^{n+m} X_{h-k}\Ulambda_k \big) =0,\quad h=0,\ldots,2n+d+m}$$
in the indeterminates $\Ulambda_k$ and with coefficients in $\BF_q$. (Each equation has $4$ linear components.) 
On the one hand, this shows that we need to compute $\alpha$ to accuracy $n'=2n+d+m+1$.  On the other hand, we see that using Gauss elimination one can solve for the unknowns in $\CO(n^{\prime\,2})$ steps where each step consists of $(4n')^2$ additions and $(4n')^2$ multiplications in the field $\BFq$. Regarding $\deg(r)$ as a structural constant and applying Proposition~\ref{PropChebtor}, the complexity is thus $\CO(n^4)$. 
\end{proof}

\begin{Rem}
We have chosen $v_0$ as a reference vertex in Theorem~\ref{PropHomCompCharOdd} for simplicity. Since $\GL_2(K_\infty)$ acts transitively on $\CT$, one could work with any reference vertex. Also, if one chooses $v_0$ as the mid point of the geodesic from $v$ to $v'$, one sees that the complexity of an algorithm to compute $\Hom_\Gamma(v,v')$ is $\CO(d^4)$ where $d=d(v,v')$. Note also that only vertices that are an even distance apart can have non-trivial $\Hom_\Gamma(v,v')$, because $d(v,\gamma v)$  is even for all $\gamma\in\Gamma$ and $v\in\CT$.
\end{Rem}

\begin{Rem}
Our implementation of algorithm of Theorem~\ref{PropHomCompCharOdd} uses the Gauss algorithm and not LLL. The linear system that needs to be solved has $4n'$ equations in $4n+2\deg(\alpha)$ variables with $n'$ as in the above proof. In practice, $\deg(\alpha)\le\deg( r)$, compare Proposition~\ref{PropChebtor}. As we shall see in Proposition~\ref{PropDiameter}, see also Remark~\ref{RemOnOptimalDiam}, we have $n\le 2\deg(r)-2$ and typically $\le 2\deg(r)-4$. Therefore we have about $4n'\le 22\deg(r)$ equations in about $10\deg(r)$ variables. Since the number of vertices of the quotient graph is essentially $q^{\deg(r)-3}$ (and $q\ge3$), already $\deg(r)=10$ is a large value. Over finite fields, systems of the size just described can be solved rather rapidly.
\end{Rem}

\begin{Rem}\label{RemOnHomGvw}
To adapt the algorithm of Theorem~\ref{PropHomCompCharOdd} to the generality proposed in Remark~\ref{RemOnGenOfAlgQuotGraph} at this point requires substantial new code for function fields. Using the notation from there, the rings $\Aa$ tend not to be UFD's and thus have a more sophisticated arithmetic. Moreover we do not expect that one should be able to give explicit simple formulas that describe the quaternion algebra $D$ and even less so a maximal order $\Lambda$ in it. One could work with non-maximal orders $\widetilde\Lambda$ but the quotient graph $\widetilde\Lambda^*\backslash\CT$ has typically much larger size than $\Lambda^*\backslash\CT$. If one has a reasonably simple description of $\Lambda$ then an algorithm as in Theorem~\ref{PropHomCompCharOdd} should be doable (certainly in the case where $S$ consists of one place only). Also, as far as we are aware of, an LLL algorithm in this generality is not implemented. Because of all these still open problems, it seems reasonable to present the algorithm here for $\BFq[T]$ only.
\end{Rem}

\section{Presentations of $\Gamma$ and the word problem}
\label{Sec-Presentations}

From a fundamental domain for the action of $\Gamma$ on $\MCT$ together with a side pairing one obtains a
 presentation of $\Gamma$ as an abstract group. This has been explained in \cite[Chapter I.4]{Serre1} interpreting $\Gamma$ as the amalgam of the stabilizers
of the vertices of $\Gamma \backslash \MCT$ along the stabilizers of the edges connecting them.
Compare also \cite[Thm.~5.7]{Papikian}.

\begin{Lem}[{{\cite[I.4.1, Lem.~4]{Serre1}}}]\label{LemSerreReprs}
Let $G$ be a group acting on a connected graph $\MCX$ and $\MCY$ a fundamental domain for the action of $G$ on $\MCX$
with an edge pairing $\EP$. Then $G$ is generated by $$\lbrace g_e \in e\in \EP \rbrace \cup 
\lbrace \Stab_G(v) \mid v \in \Ver(\MCS) \rbrace.$$
\end{Lem}

The relations among the generators of the previous lemma are given by \cite[\S~I.5, Thm.~13]{Serre1} and based on 
the construction of the fundamental group $\pi(\Gamma, \MCY, \MCS)$ in \cite[p.~42]{Serre1}.
For the group $\Gamma$ considered here, all non-terminal vertices $v$ of $\MCS$ have stabilizer $\BF_q^*$ which lies in the center of $\Gamma$. The results just quoted therefore considerably simplify and yield:

\begin{Prop}\label{PropGenRel}
Let $(\CY,\CS,(g_e)_{e\in\EP})$ be a fundamental domain with an edge pairing for $(\Gamma,\CT)$ as provided by Algorithm~\ref{AlgQuotGraphNew}. For each terminal vertex $v\in\Ver(\CS)$, let $g_v$ be a generator of $\Stab_{\Gamma}(v)$. Then 
$\Gamma$ is isomorphic to the group generated by 
$$ \lbrace g_0 \rbrace \cup \lbrace g_v \mid v \text{ terminal in } \Ver(\CS) \rbrace 
\cup \lbrace g_e \text { the edge-label} \mid e \in \EP \rbrace$$
subject to the relations
$$g_0^{q - 1} = 1, \text{ } g_v^{q + 1} = g_0 \text{ for all terminal } v, \text{ }
 [g_e, g_0] = 1 \text{ for all } e\in \EP.$$
\end{Prop}

In particular $g_0$ lies in the center of $\Gamma$, as it should.

\begin{Ex}
In Example~\ref{ExQuotGraph} the group $\Gamma$ is generated by
$$\lbrace g_0, g_{v_1}, \dots, g_{v_8}, g_1, \dots, g_5 \rbrace$$
with relations
$$g_0^4 = 1, g_{v_i}^{6} = g_0, [g_0, g_i] = 1.$$
\end{Ex}

The word problem with respect to this set of generators was already solved
by the reduction Algorithm~\ref{AlgReduction}, compare \cite[Remark 4.6]{Voight}.

\section{Complexity analysis and degree bounds}
\label{Sec-Complex}

In this section we will analyze the complexity of Algorithm~\ref{AlgQuotGraphNew} and obtain
some bounds on the size of generators of $\Gamma$. We start by bounding the diameter of the graph 
$\Gamma \backslash \MCT$. The idea of using the Ramanujan property to obtain complexity bounds
was inspired by \cite[Conj.~6.6]{Kirschmer-Voight}. A standard reference is \cite{Lubotzky}.

\begin{Def}
 A $k$-regular connected graph $\MCG$ is called a {\em Ramanujan graph} if for every eigenvalue $\lambda$ of the
adjacency matrix of $\CG$ either $\lambda = \pm k$ or $| \lambda | \leq 2 \sqrt{k - 1}$.
\end{Def}

\begin{Prop}[{{\cite[Prop~7.3.11]{Lubotzky}}}]\label{ThmRamanujan}
 Let $\MCG$ be a $k$-regular Ramanujan graph on $n\ge3$ vertices.\footnote{The proof in \cite{Lubotzky} requires at least one eigenvalue $\lambda$ of the adjacency matrix with $| \lambda | \leq 2 \sqrt{k - 1}$ and hence $n\ge3$. Also, the assertion is obviously wrong for $n=2$ and $k$ large.}  Then 
$$\diam(\MCG) \leq \log_{k - 1}(4n^2).$$
\end{Prop}

Let $$\one(R) := \begin{cases} 1 & \text{ if some place in $R$ has degree one,}\\ q(q - 1) & \text{ otherwise.} \end{cases}$$

\begin{Lem}\label{LemCovering}
There is a covering of $\MCG := \Gamma \backslash \MCT$  by a $q + 1$-regular 
Ramanujan graph $\tilde \MCG$ with
 $$\# \Ver(\tilde \MCG) = \frac{2 \one(R)}{(q - 1)^2} \prod_{\Fp \in R} (q_\Fp - 1).$$
\end{Lem}

\begin{proof}
Recall the definitions and formulas for $V_1$ and $V_{q + 1}$ from Theorem~\ref{ThmQuotStr}.  
If $\one(R) = 1$, we can choose a degree $1$ place $\Fp_0 \in R$. If not we choose an arbitray
degree $1$ prime $\Fp_0$. Let $\Gamma(\Fp_0)$ be the full level $\Fp_0$ congruence
subgroup in $\Gamma$. By \cite[Thm.~1.2]{Lubotzky2} we know that $\tilde \MCG := (\Gamma \cap \Gamma(\Fp_0)) \backslash \MCT$ 
is a Ramanujan graph. Observe that $(\Gamma \cap \Gamma(\Fp_0)) \backslash \Gamma \cong \BF^*_{q^2}$ if $\Fp_0 \in R$, which 
has cardinality $q^2 - 1$, and $(\Gamma \cap \Gamma(\Fp_0)) \backslash \Gamma \cong \GL_2(\BFq)$ otherwise, which has cardinality
$\one(R)(q^2 - 1)$.
By analyzing the growth of the stabilizers from $\Gamma \cap \Gamma(\Fp_0)$ to $\Gamma$, we observe that
$$\frac{1}{\one(R)} \# \Ver(\tilde \MCG) = V_1 + (q + 1) V_{q + 1}$$ $$=  
\Big(V_1 + \frac{q + 1}{q - 1}V_1 + \frac{2(q + 1)}{q - 1}(g(R) - 1)\Big)
=\frac{2}{(q - 1)^2} \prod_{\Fp \in R} (q_\Fp - 1).$$
\end{proof}

\begin{Prop}\label{PropDiameter}
Suppose $\Ver(\Gamma \backslash \MCT)\ge3$. Then 
$$\diam(\Gamma \backslash \MCT) \leq 2 \deg(r) + 2(2 \log_q(2) + 1 - \log_q(q - 1)).$$
\end{Prop}

\begin{proof}
Let $\MCG = \Gamma \backslash \MCT$ and $\MCG^\prime$ be the covering from Lemma~\ref{LemCovering}. Then
\begin{eqnarray*}
 \diam(\MCG) \ \leq \ \diam(\MCG^\prime) &\stackrel{\ref{ThmRamanujan}}\leq &2 \log_q(\#\Ver(\MCG^\prime)) + \log_q(4)\\
&\stackrel{\ref{LemCovering}}\leq &2 \log_q \Big(\frac{2q}{q - 1} \prod_{\Fp \in R} (q_\Fp - 1)\Big) + \log_q(4)\\ 
&\leq& 4 \log_q(2) + 2 \log_q(\frac{q}{q-1}) + 2 \log_q \big(\prod_{\Fp \in R}q_\Fp \big)
\\ &=& 2(2 \log_q(2) + 1 - \log_q(q - 1)) + 2 \deg(r).
\end{eqnarray*}
\end{proof}

\begin{Cor}\label{CorSize}
With $\|\,\|$ as in (\ref{DefOfHeight}), the group $\Gamma$ is generated by the set 
$$\lbrace \gamma \in \Gamma \mid \| \gamma \| \leq \deg(\alpha)/2 + 2 \deg(r) + 2(2 \log_q(2) + 1 - \log_q(q - 1)) \rbrace.$$
\end{Cor}

\begin{proof}
By Proposition~\ref{PropGenRel}, the group $\Gamma$ is generated by the vertex and edge labels of the quotient graph from
Algorithm~\ref{AlgQuotGraphNew}. By Proposition~\ref{PropHomCompCharOdd} these labels $g_t$ have norm 
$\| g_t \| \leq \deg(\alpha)/2 + n$, where $n$ is the distance in $\Gamma \backslash \mathcal T$ between the 
initial vertex and the labeled vertex. In particular, $n\le \diam(\Gamma \backslash \MCT)$.  
\end{proof}

\begin{Rem}\label{RemOnOptimalDiam}
 If $\one(R) = 1$, we can obviously subtract $2 + 2 \log_q(q - 1)$ from the diameter in Proposition~\ref{PropDiameter} and subsequently from the
bounds in Corollary~\ref{CorSize}. In the other case we expect this to be possible as well. This should follow by replacing 
$\Gamma(\Fp_0)$ by $$\tilde \Gamma_1(\Fp_0) := \lbrace \gamma \in \Gamma \mid \gamma \equiv 
\begin{pmatrix} 1 & \star \\ 0 & \star \end{pmatrix} \pmod {\Fp_0} \text{ in } \GL_2(\BFq) \rbrace.$$
Unfortunately we could not find this analog of \cite[Thm.~1.2]{Lubotzky2} for a congruence
subgroup other than $\Gamma(\Fp)$ in the literature although it seems likely to hold. 

If this was indeed true, we would obtain the improved bound 
$$\diam(\Gamma \backslash \MCT) \leq 2 \deg(r) + 4 \log_q(2) - 4 \log_q(q - 1).$$ 
For $q > 19$ it gives $\diam(\Gamma \backslash \MCT) \leq 2 \deg(r) - 4$. 
The nice feature of this last bound is that it was assumed in many concrete examples that we have computed.
\end{Rem}

\begin{Prop}
 Algorithm~\ref{AlgQuotGraphNew} computes the quotient graph $\Gamma \backslash \MCT$ in time 
$$\MCO((\#\Ver(\Gamma \backslash \MCT))^2 \diam(\Gamma \backslash \MCT)^5)\stackrel{\ref{ThmQuotStr}}=\CO(q^{2\deg(r)-6}\cdot\deg(r)^5)$$
in terms of operations over $\BFq$.
\end{Prop}

\begin{proof}
 According to Prop~\ref{PropHomCompCharOdd}, comparing two vertices in the algorithm can be done in time $\MCO(n^4)$,
where $n$ is always less or equal then $\diam(\Gamma \backslash \MCT)$. The list of vertices in each step of the algorithm
is always shorter than the cardinality of $\Ver(\Gamma \backslash \MCT)$, so in each step the number of
comparisons is bounded by $(\# \Ver(\Gamma \backslash \MCT))^2$. The number of steps is bounded by $\diam(\Gamma \backslash \MCT)$ and the result follows.
\end{proof}

\end{document}